\documentclass[11pt]{amsart}
\usepackage{latexsym, amsmath, amssymb,amsthm,amsopn,amsfonts}
\usepackage{version}
\usepackage{epsfig,graphics,color,graphicx,graphpap}
\usepackage{amssymb}
\usepackage{todonotes}
\usepackage{listings}
\usepackage{mathrsfs}
\usepackage[citecolor=blue]{hyperref}

\usepackage[framemethod=tikz]{mdframed}
\newcommand{\redsout}{\bgroup\markoverwith{\textcolor{red}{\rule[0.5ex]{2pt}{.4pt}}}\ULon}

\newcommand{\LV}{\left|}
\newcommand{\RV}{\right|}
\newcommand{\LC}{\left(}
\newcommand{\RC}{\right)}

\newcommand{\p}{\partial}

\numberwithin{equation}{section}
\newtheorem{theorem}{Theorem}[section]

\newtheorem{proposition}[theorem]{Proposition}
\newtheorem{lemma}[theorem]{Lemma}

\newtheorem{remark}{Remark}[section]
 
\newcommand{\R}{\mathbb R}

\usepackage{url}
\definecolor{mycolor}{rgb}{0.122, 0.435, 0.698}
\definecolor{aliceblue}{rgb}{0.94, 0.97, 1.0}

\newmdenv[innerlinewidth=0.5pt, roundcorner=4pt,linecolor=mycolor,innerleftmargin=6pt,
innerrightmargin=6pt,innertopmargin=6pt,innerbottommargin=6pt]{mybox}
\newmdenv[backgroundcolor=aliceblue,innerlinewidth=0.5pt, roundcorner=4pt,linecolor=mycolor,innerleftmargin=6pt,
innerrightmargin=6pt,innertopmargin=6pt,innerbottommargin=6pt]{mybox1}
\author[Lai]{Ru-Yu Lai}
\address{School of Mathematics, University of Minnesota, Minneapolis, MN 55455, USA}
\email{rylai@umn.edu}

\author[Uhlmann]{Gunther Uhlmann}
\address{Department of Mathematics, University of Washington, Seattle, WA 98195, USA; HKUST Jockey Club Institute for Advanced Study, HKUST, Clear Water Bay, Kowloon, Hong Kong}
\email{gunther@math.washington.edu}

\author[Zhou]{Hanming Zhou}
\address{Department of Mathematics, University of California Santa Barbara, Santa Barbara, CA 93106, USA}
\email{hzhou@math.ucsb.edu}

 
\title{Recovery of coefficients in semilinear transport equations}    
\begin{document}

\maketitle
 
\begin{abstract}
We consider the inverse problem for time-dependent semilinear transport equations. We show that time-independent coefficients of both the linear (absorption or scattering coefficients) and nonlinear terms can be uniquely determined, in a stable way, from the boundary measurements by applying a linearization scheme and Carleman estimates for the linear transport equations. We establish results in both Euclidean and general geometry settings.
\end{abstract}

\section{Introduction}\label{sec:intro}
We investigate the time-dependent transport equation with nonlinear term in this article.
Let $\Omega\subset\R^d$, $d\geq 2$, be an open bounded and convex domain with smooth boundary $\p\Omega$. We denote
$$
S\Omega:=\Omega\times \mathbb{S}^{d-1},\quad S\Omega^2:=\Omega\times \mathbb{S}^{d-1}\times \mathbb{S}^{d-1}, \quad S\Omega_T:=(0,T)\times S \Omega
$$
for $T>0$. We also denote the outgoing and incoming boundaries of $S\Omega$ by $\p_+ S \Omega$ and $\p_- S \Omega$ respectively which are defined as follows:
\begin{align*}
	\p_\pm S\Omega := \{(x,v)\in S\Omega:\, x\in \p \Omega,\, \pm \langle n(x), v\rangle > 0\},
\end{align*}
where $n(x)$ is the unit outer normal vector at $x\in\p \Omega$ and $\langle v,w \rangle$ is the dot product in $\R^d$. Moreover, $\p_\pm S\Omega _T:=(0,T)\times \p_\pm S\Omega.$
Let the function $f\equiv f(t,x,v)$ be the solution to the following initial boundary value problem for the nonlinear transport equation:
\begin{align}\label{EQN: RTE equ intro}
	\left\{\begin{array}{rcll}
		\p_tf + v\cdot \nabla_x f + \sigma f + N(x,v,f) &=& K(f) & \hbox{in } S\Omega_T,  \\
		f &=&f_0 & \hbox{on }\{0\}\times S\Omega,\\
		f  &=&f_- & \hbox{on }\p_-S\Omega_T,
	\end{array}\right.
\end{align}
where $T$ is sufficiently large, $\sigma\equiv \sigma(x,v)$ is the absorption coefficient and the scattering operator $K$ takes the form
\begin{align}\label{def:collision}
	K(f)(t,x,v) := \int_{\mathbb{S}^{d-1}} \mu(x,v',v)f(t,x,v') \,d\omega(v'),
\end{align}
with the scattering coefficient $\mu\equiv\mu(x,v',v)$ and the normalized measure, that is, $\int_{\mathbb{S}^{d-1}} \,d\omega(v')=1$, where $d\omega(v')$ is the measure on $\mathbb{S}^{d-1}$.

In this paper, we are interested in the inverse problem for the nonlinear transport equation in \eqref{EQN: RTE equ intro}. The main objectives are to determine the nonlinearity $N$, absorption $\sigma$ and the scattering coefficient $\mu$ by the boundary data. The problem is motivated by applications in the photoacoustic tomography, in which the nonlinear excitation is observed due to two-photon absorption effect of the underlying medium, see \cite{LaReZh21,RenZhong21, RenZhang2018, StZh-arXiv21} and the references therein.

There have been extensive study in the inverse coefficient problem for the transport equation. The associated inverse problem is concerned with determining unknown properties (such as absorption and scattering coefficients, $\sigma$ and $\mu$) from the albedo operator which maps from incoming to outgoing boundary.
The uniqueness result was studied in \cite{IKun, CS1, CS2, CS3, CS98, SU2d} and stability estimates were derived in \cite{Bal14, Bal10,  Bal18,BalMonard_time_harmonic, Machida14, Wang1999, ZhaoZ18}. See also recent references \cite{Balreview, Stefanov_2003}. Moreover, related studies in the Riemannian setting can be found in \cite{AY15, MST10, MST10stability, MST11, McDowall04}.
As for the nonlinear transport equation, the unique determination for the kinetic collision kernel was derived in \cite{LaiUhlmannYang} for the stationary Boltzmann equation and in \cite{LiOuyang2022} for the time-dependent Boltzmann equation. In addition to the recovery of the collision kernel, the determination of the Lorentzian spacetime, i.e. the first order information, from the source-to-solution map for the Boltzmann equation was considered in \cite{BKLL21}. 

The main strategy we applied here is using the Carleman estimate for the linear transport equation and the linearization technique. A Carleman estimate, established by Carleman \cite{carleman}, is an $L^2$ weighted estimate for a solution to a partial differential equation with large parameters. Roughly speaking, a special weight function in the Carleman estimate is chosen to control irrelevant information and then extract the desired properties. The Carleman estimates have been successfully applied in solving inverse problems for various equations.
We refer the readers to the related references \cite{BK1981, Gaitan14,  Yamamoto2016, KlibanovP2006, Klibanov08, Lai-L, Machida14} for the application in the inverse transport problem.
As for the linearization technique, it deals with nonlinear equations in inverse problems to reduce the nonlinear equation to the linear one. 
In this paper, we apply the higher order linearization whose feature is that it introduces small parameters into the problem for the nonlinear equation. Then differentiating it multiple times with respect to these parameters to earn simpler and new linearized equations. For more detailed discussions and related studies, see for instance \cite{CLOP,KLU2018,LUW2018} for hyperbolic equations, \cite{FL2019, Kang2002, KU201909, KU2019,LaiOhm21, LaiTingZhou20, LaiTingZhou21, LaiLin20, LLLS201903, LLLS201905, LiLiSaTy-arxiv} for elliptic equations, and \cite{LaReZh21, LaiUhlmannYang, LiOuyang2022} for kinetic equations.

\subsection{Main results}
Throughout this paper, we suppose that $T$ is sufficiently large which depends on the domain. Suppose that $\sigma \in L^\infty(S\Omega)$ and $\mu\in L^\infty(S\Omega^2)$ and
there exist positive constants $\sigma^0$ and $\mu^0$ such that 
\begin{align}\label{EST:sigma}
    0\leq\sigma(x,v) \leq \sigma^0, \quad 0 \leq \mu(x,v',v) \leq \mu^0.
\end{align}
Moreover, suppose that $\mu$ satisfies 
\begin{align}\label{EST:k}
    \int_{\mathbb{S}^{d-1}} \mu(x,v,v') \,d\omega(v')\leq \sigma(x,v), \quad\hbox{and}\quad \int_{\mathbb{S}^{d-1}} \mu(x,v',v) \,d\omega(v')\leq \sigma(x,v) 
\end{align}
for almost every $(x,v)\in S\Omega$.  
The assumption \eqref{EST:k} means that the absorption effect is stronger than the scattering effect in the medium.

Now we denote the measurement operator $\mathcal{A}_{\sigma, \mu,N}$ by 
\begin{align}\label{bdry operator}
	\mathcal{A}_{\sigma, \mu,N}: (f_0,f_-)\in L^\infty(S\Omega)\times L^\infty(\p_-S\Omega_T)\mapsto f|_{\p_+S\Omega_T}\in L^\infty(\p_+S\Omega_T).
\end{align}
It follows from Theorem~\ref{THM:WELL} in Section~\ref{sec:Preliminaries} that the initial boundary value problem \eqref{EQN: RTE equ intro} is well-posed for small initial and boundary data $(f_0,f_-)$. 
Specifically, there exists a small parameter $\delta>0$ such that when  
\begin{equation}\label{DEF:set X}
	\begin{split}
	(f_0,f_-) \in \mathcal{X}^\Omega_\delta:=  \{(f_0,f_-)\in L^\infty(S\Omega)\times L^\infty(\p_-S\Omega_T):\, \|f_0\|_{L^\infty(S\Omega)}\leq \delta,\,  \|f_-\|_{L^\infty(\p_-S\Omega_T)}\leq \delta\},
	\end{split}
\end{equation}
the initial boundary value problem \eqref{EQN: RTE equ intro} has a unique solution. 
Hence, the map $\mathcal{A}_{\sigma, \mu,N}$ is well-defined within the class of small given data.

The paper is devoted to investigating the inverse coefficient problem for the transport equation with nonlinearity. We study the reconstruction of the absorption coefficient (or scattering coefficient) as well as the nonlinear term from the measurement operator.
In the following, we illustrate the main results on $\R^d$ (discussed in Section~\ref{sec:Euclidean}) and also results on Riemannian manifolds (discussed in Section~\ref{sec:Riemannian}) separately.

\subsubsection{Inverse problems in Euclidean space}
In the first theme of the paper, we consider the problem \eqref{EQN: RTE equ intro} with the nonlinear term $N(x,v,f):S\Omega\times \R\rightarrow\R$ satisfying the following conditions:\\
\begin{align}\label{condition N}
\left\{ \begin{array}{ll}
&\hbox{the map $z\mapsto N(\cdot,\cdot,z)$ is analytic on $\R$ such that $N(\cdot,\cdot,f)\in L^\infty(S\Omega)$}; \\ 
&\hbox{$N(x,v,0)=\p_zN(x,v,0)=0$ in $S\Omega$. }\end{array}\right. 
\end{align}
This implies that $N$ can be expanded into a power series
\begin{align}\label{condition N Taylor}
N(x,v,z)= \sum_{k=2}^\infty q^{(k)}(x,v){z^k\over k!} , 
\end{align} 
which converges in the $L^\infty(S\Omega)$ topology with $q^{(k)}(x,v):=\p_z^kN(x,v,0)\in L^\infty(S\Omega)$.

For a fixed vector $\gamma\in\mathbb{S}^{d-1}$, we say a function $p$ is in the set $\Lambda$ if $p$ satisfies 
\begin{align}\label{condition M}
     p(x,v)=p(x,-v)\quad\hbox{ in }S\Omega \quad\hbox{and}\quad p(x,v)=0\quad 
    \hbox{ in }\Omega\times \{v\in\mathbb{S}^{d-1}:\, |\gamma\cdot v|\leq \gamma_0\} 
\end{align}
for some fixed constant $\gamma_0>0$.
We state the first main result. The inverse problem here is to recover $\sigma$ and $N$ provided that $\mu$ is given.

\begin{theorem}\label{thm:into Euclidean}
	Let $\Omega$ be an open bounded and convex domain with smooth boundary. 
	Suppose that $\sigma_j\in L^\infty(S\Omega)$ and $\mu\in L^\infty(S\Omega^2)$ satisfy \eqref{EST:sigma} and \eqref{EST:k} for $j=1,\,2$. 
	Let $N_j:S\Omega\times \R\rightarrow\R$ satisfy the assumption \eqref{condition N} with $q^{(k)}$ replaced by $q^{(k)}_j$ for $j=1,2$, respectively.
	Let $\sigma_j$, $\mu(\cdot,\cdot,v)$, $q_j^{(k)}$ for all $k\geq 2$ be in $\Lambda$.
    If $$\mathcal{A}_{\sigma_1,\mu,N_1}(h,0)=\mathcal{A}_{\sigma_2,\mu,N_2}(h,0)$$ 
	for any $h\in L^\infty(S\Omega)$ with $\|h\|_{L^\infty(S\Omega)}\leq \delta$ for sufficiently small $\delta$, 
	then 
	$$
	\sigma_1(x,v)=\sigma_2(x,v)  \quad\hbox{ in }S\Omega  \quad \hbox{and}\quad 
	N_1(x,v,z) = N_2(x,v,z) \quad\hbox{ in }S\Omega\times \R.
	$$
\end{theorem}

\begin{remark}
    We would like to point out that the constant $\gamma_0$ indeed can be chosen to be arbitrarily small as long as $\gamma_0>0$. In this case, the condition \eqref{condition M} becomes less restrictive in the sense that the coefficients only need to vanish in a small subset of $\mathbb{S}^{d-1}$ in order to make the above uniqueness results hold. 
    We refer to Section~\ref{sec:Euclidean} for detailed discussions and for more relaxed conditions, instead of \eqref{condition M}, on $\sigma_j,\, \mu$ and $q_j^{(k)}$.
\end{remark}
\begin{remark}\label{RK:intro 1}
	On the other hand, suppose that $\sigma$ is given and $\mu$ is unknown and is of the form $\mu:=\tilde{\mu}(x,v)p(x,v',v)$. In this case, we can also recover
	$\tilde\mu$, see Proposition~\ref{thm:recover mu} for details. Combining with the reconstruction of $N(x,v,z)$, we obtain the determination of both the scattering coefficient and the nonlinear term provided that $\sigma$ is known.   
\end{remark}

Moreover, we also consider the problem when the nonlinear term has the form
\[
N(x,v,f)= q(x,v)N_0(f),
\] 
where $N_0$ satisfies
\begin{align}\label{EST:N term intro}
    \|N_0(f)\|_{L^\infty(S\Omega_T)}\leq C_1 \|f\|_{L^\infty(S\Omega_T)}^\ell,
\end{align} 
and
\begin{align}\label{EST:N intro}
	\|\p_z N_0(f)\|_{L^\infty(S\Omega_T)}\leq C_2 \|f\|_{L^\infty(S\Omega_T)}^{\ell-1}
\end{align}
for a positive integer $\ell\geq 2$ and constants $C_1,C_2>0$, independent of $f$. 
For instance, when $\ell=2$, $N_0(f)$ can represent the quadratic nonlinearity such as $N_0(f)=f^2$ or $f\int_{\mathbb{S}^{d-1}} fd\omega(v')$.
The latter example finds applications in photoacoustic tomography with nonlinear absorption effect and we refer the interested readers to the references \cite{LaReZh21,RenZhang2018}. 
\begin{theorem}\label{thm:into Euclidean N2}
	Let $\Omega$ be an open bounded and convex domain with smooth boundary. 
	Suppose that $\sigma_j\in L^\infty(S\Omega)$ and $\mu\in L^\infty(S\Omega^2)$ satisfy \eqref{EST:sigma} and \eqref{EST:k} for $j=1,\,2$. 
	Let $N_j(x,v,f)=q_j(x,v)N_0(f)$, where $q_j\in L^\infty(S\Omega)$ for $j=1,2$ and $N_0$ satisfies \eqref{EST:N term intro}-\eqref{EST:N intro} with $\p^2_z N_0(0)>0$. Let $\sigma_j$, $\mu(\cdot,\cdot,v)$, $q_j$ be in $\Lambda$.
    	If 
    	$$\mathcal{A}_{\sigma_1,\mu,N_1}(h,0)=\mathcal{A}_{\sigma_2,\mu,N_2}(h,0)$$ 
    	for any $h\in L^\infty(S\Omega)$ with $\|h\|_{L^\infty(S\Omega)}\leq \delta$ for sufficiently small $\delta$, then 
    $$
	\sigma_1(x,v)=\sigma_2(x,v)\quad\hbox{in }S\Omega\quad \hbox{and}\quad N_1(x,v,z)=N_2(x,v,z)\quad \hbox{in }S\Omega\times \R.
	$$
\end{theorem}

\begin{remark}
Similarly, as discussed in Remark~\ref{RK:intro 1}, if $\sigma$ is now given, then we can recover $\tilde{\mu}$ and $N$ from the boundary data as well.
\end{remark}

\subsubsection{Inverse problems on manifolds}
The second theme of the paper is the inverse problems for the transport equation on manifolds. 

We denote $M$ the interior of a compact non-trapping Riemannian manifold $(\overline M, g)$ with smooth strictly convex (with respect to the metric $g$) boundary $\p M$. Since $\overline M$ is non-trapping, any maximal geodesic will exit $\overline M$ in finite time, i.e. have finite length. $M$ plays the role of $\Omega$ in the manifold case, and thus we naturally generalize the notations for $\Omega$ (e.g. $S\Omega$, $S\Omega_T$, $\p_\pm S\Omega_T$, etc.) to corresponding notations for $M$ (e.g. $SM$, $SM_T$, $\p_\pm SM_T$, etc.). See Section 2 for more details.

We consider the following initial boundary value problem:
\begin{align}\label{EQN: RTE equ M}
	\left\{\begin{array}{rcll}
		\p_tf + X f + \sigma f + N(x,v,f) &=& 0& \hbox{in } SM_T,  \\
		f &=&f_0 & \hbox{on }\{0\}\times SM,\\
		f &=&f_- & \hbox{on }\p_-SM_T.
	\end{array}\right.
\end{align}
Here $X$ is the geodesic vector field which generates the geodesic flow on $SM$, see section \ref{sec:Preliminaries} for more details. In particular, $X=v\cdot \nabla_x$ in the Euclidean case.
The equation \eqref{EQN: RTE equ M} is in the absence of the scattering effect, due to our Carleman estimates on Riemannian manifolds in Section~\ref{sec:Riemannian}. The Carleman weight function chosen in this paper is naturally associated with the geodesic flow of the Riemannian manifold, which depends on both the position $x$ and the direction $v$, and therefore makes it hard to control the scattering term by other terms in the estimate, see also Remark~\ref{remark no scattering}. Since the main scope of the paper is recovering the nonlinearity of the transport equation, we do not pursue further the inverse problem with the scattering term in the Riemannian case.

Let $\mathcal A_{\sigma,N}:=\mathcal A_{\sigma,0,N}$ be the measurement operator associated with the problem \eqref{EQN: RTE equ M}. Analogous to the results in the Euclidean case, we have the following two main results on Riemannian manifolds.
\begin{theorem}\label{thm:into Riemannian 1}
	Let $M$ be the interior of a compact non-trapping Riemannian manifold $\overline M$ with smooth strictly convex boundary $\p M$. Suppose that $\sigma_j\in L^\infty(SM)$ satisfy \eqref{EST:sigma} for $j=1,\,2$. 
	Let $N_j:SM\times \R\rightarrow\R$ satisfy the assumption \eqref{condition N} in the manifold with $q^{(k)}$ replaced by $q^{(k)}_j$ for $j=1,2$, respectively.
	If 
	$$\mathcal{A}_{\sigma_1, N_1}(h,0)=\mathcal{A}_{\sigma_2, N_2}(h,0)$$ 
	for any $h\in L^\infty(SM)$ with $\|h\|_{L^\infty(SM)}\leq \delta$ for sufficiently small $\delta$, then 
	$$
	\sigma_1(x,v)=\sigma_2(x,v)  \quad\hbox{ in }SM  \quad \hbox{and}\quad 
	N_1(x,v,z) = N_2(x,v,z) \quad\hbox{ in }SM\times \R.
	$$
\end{theorem}

Moreover, when the nonlinear term takes the form $N(x,v,f)=q(x,v)N_0(f)$, we have the following result.
\begin{theorem}\label{thm:into Riemannian}
	Suppose that $\sigma_j\in L^\infty(SM)$ satisfies \eqref{EST:sigma} for $j=1,\,2$. 
	Let $N_j(x,v,f)=q_j(x,v)N_0(f)$, where $q_j\in L^\infty(SM)$ for $j=1,2$ and $N_0$ satisfies \eqref{EST:N term intro}-\eqref{EST:N intro} in the manifold with $\p^2_z N_0(0)>0$.
	If 
	$$\mathcal{A}_{\sigma_1, N_1}(h,0)=\mathcal{A}_{\sigma_2, N_2}(h,0)$$ 
	for any $h\in L^\infty(SM)$ with $\|h\|_{L^\infty(SM)}\leq \delta$ for sufficiently small $\delta$, then
	$$
	\sigma_1(x,v)=\sigma_2(x,v)  \quad\hbox{ in }SM  \quad \hbox{and}\quad 
	N_1(x,v,z) = N_2(x,v,z) \quad\hbox{ in }SM\times \R.
	$$
\end{theorem}

\begin{remark}
    We actually only need much less data to stably determine both $\sigma$ and $N$. To be more specific, fix positive $h\in L^\infty(SM)$ with $X^\beta h\in L^\infty(SM)$ for $\beta=1,2$, consider the initial boundary value condition $(\varepsilon h,0)$ for $|\varepsilon|$ sufficiently small, we establish the following stability result
    \begin{align*}
        \|\sigma_1-\sigma_2\|_{L^2(SM)} & \leq C\|\p_t\p_\varepsilon \big(\mathcal A_{\sigma_1,N_1}(\varepsilon h,0)-\mathcal A_{\sigma_2,N_2}(\varepsilon h,0)\big)|_{\varepsilon=0}\|_{L^2(\p_+SM_T)}.
        \end{align*}
    If in addition $\sigma=\sigma_1=\sigma_2$, then
    \begin{align*}
        \|q_1-q_2\|_{L^2(SM)} & \leq C\|\p_t\p^2_\varepsilon \big(\mathcal A_{\sigma,N_1}(\varepsilon h,0)-\mathcal A_{\sigma,N_2}(\varepsilon h,0)\big)|_{\varepsilon=0}\|_{L^2(\p_+SM_T)}.
    \end{align*}
    The constants $C$ in both estimates are independent of $\sigma_j$ and $q_j$, $j=1,2$. See Proposition~\ref{thm:recover sigma Riemannian} and Proposition~\ref{thm:recover q Riemannian} for more details. Similar results hold when the nonlinear term $N_j$, $j=1,2$ satisfy the assumption \eqref{condition N}, see e.g. Lemma \ref{lemma:q2} and the proof of Lemma \ref{lemma:q m}.
\end{remark}

The remaining part of the paper is organized as follows. In Section~\ref{sec:Preliminaries}, we introduce the notations and function spaces, and also establish several preliminary results, including boundedness of solutions to the linear equation, Maximum principle, and the well-posedness problem for the nonlinear transport equation. We investigate the reconstruction of the unknown coefficients in the Euclidean setting and prove Theorem~\ref{thm:into Euclidean} in Section~\ref{sec:Euclidean}. In particular, we establish an improved version of the Carleman estimate of \cite{Machida14}. In Section~\ref{sec:Riemannian}, we first deduce the Carleman estimate and the energy estimate in a Riemannian manifold. With these estimates, Theorem~\ref{thm:into Riemannian 1} follows directly by applying similar arguments as in the proof of Theorem~\ref{thm:into Euclidean}. Furthermore, in the case of $N =q N_0(f)$, we show the unique determination of $q$, which immediately implies the uniqueness of $N$ in Theorem~\ref{thm:into Riemannian}. Finally, we note that the techniques for showing Theorem~\ref{thm:into Riemannian} can also be applied to prove Theorem~\ref{thm:into Euclidean N2}. 


\section{Preliminaries}\label{sec:Preliminaries}
In this section, we will discuss the forward problem for the initial boundary value problem for the nonlinear transport equation. In particular, we will prove the well-posedness result on a more general setting, namely, the Riemannian manifold. All the results discussed in this section are also valid in the Euclidean space and will be utilized in Section~\ref{sec:Euclidean}.

\subsection{Notations and spaces} 
In order to investigate the transport equation on a Riemannian manifold, we need to introduce the related notations first. Most of the notations below are similar to the ones we saw earlier in Section~\ref{sec:intro}, but with $\Omega$ replaced by the manifold $M$.
 
Let $M$ be the interior of a compact Riemannian manifold $(\overline M, g)$, of dimension $d\geq 2$, with a Riemannian metric $g$ and strictly convex boundary $\p M$. Suppose that $\overline M$ is non-trapping.
Let $TM$ be the tangent bundle of $M$. We denote the unit sphere bundle of the manifold $(M,g)$ by 
$$
SM :=\{(x,v)\in TM:\, |v|^2_{g(x)}:=\left< v,v\right>_{g(x)}=1\},
$$
where $\langle \cdot\,,\,\cdot \rangle_{g(x)}$ is the inner product on the tangent space $T_xM$.
Let $\p_+ S M$ and $\p_- SM$ be the outgoing and incoming boundaries of $SM$ respectively and they are defined by
\begin{align*}
	\p_\pm SM := \{(x,v)\in SM:\, x\in \p M,\, \pm \langle n(x), v\rangle_{g(x)} > 0\},
\end{align*}
where $n(x)$ is the unit outer normal vector at $x\in\p M$.
For any point $x\in M$, let $S_xM:=\{v:\, (x,v)\in SM\}$. Moreover, we also denote 
$$S M^2:=\{(x,v, v'): x\in M,\; v,\, v'\in S_xM \}.$$
Let $T>0$, we denote $SM_T:=(0,T)\times SM$ and $\p_\pm SM_T:=(0,T)\times \p_\pm SM$.

For every point $x\in M$ and every vector $v\in S_xM$, let $\gamma_{x,v}(s)$ be the maximal geodesic satisfying the initial conditions 
$$
\gamma_{x,v}(0)=x,\quad \dot \gamma_{x,v}(0)=v.
$$
Since $M$ is non-trapping, $\gamma_{x,v}$ is defined on the finite interval $[-\tau_-(x,v),\tau_+(x,v)]$. Here the two travel time functions
\begin{align}
	\tau_\pm : SM\rightarrow [0,\infty)
\end{align}
are determined by $\gamma(\pm\tau_\pm(x,v))\in \p M$. In particular, they satisfy 
$\tau_+(x,v)=\tau_-(x,-v)$ for all $(x,v)\in SM$ and 
$\tau_-(x,v)|_{\p_-SM}=\tau_+(x,v)|_{\p_+SM}=0$.
Denote the geodesic flow by $$\phi_t(x,v)=(\gamma_{x,v}(t),\dot\gamma_{x,v}(t)).$$
Let $X$ be the generating vector field of the geodesic flow $\phi_{t}(x,v)$, that is, for a given function $f$ on $SM$, $Xf(x,v)=\frac{d}{dt}f(\phi_t(x,v))|_{t=0}$. Notice that in the Euclidean space $\R^d$, $\phi_t(x,v)=(x+tv,\,v)$ and $X=v\cdot\nabla_x$ where $v$ is independent of $x$.

We define the spaces $L^p(SM)$ and $L^p(SM_T)$, $1\leq p<\infty$, with the norm
$$
    \|f\|_{L^p(SM)}= \LC\int_{SM} |f|^p\,d{\Sigma}\RC^{1/p} \quad\hbox{and}\quad \|f\|_{L^p(SM_T)}= \LC\int^T_0\int_{SM} |f|^p\,d{\Sigma}dt\RC^{1/p},
$$ 
with $d\Sigma=d\Sigma(x,v)$ the volume form of $SM$.
Moreover, for the spaces $L^p(\p_\pm SM_T)$, we define its norm to be
$$
\|f\|_{L^p(\p_\pm SM_T)}=\|f\|_{L^p(\p_\pm SM_T; \pm d\xi})= \LC\int^T_0\int_{\p_\pm SM} |f|^p\, (\pm d\xi) dt\RC^{1/p},
$$ 
where $d\xi(x,v):=\langle n(x),v\rangle_{g(x)}d\tilde{\xi}(x,v)$ with $d\tilde{\xi}$ the standard volume form of $\p SM$. Note that in the Euclidean setting since $v$ is independent of $x$, we denote $d\tilde\xi = d\lambda(x)d\omega(v)$, where $d\lambda$ is the measure on $\p\Omega$ and $d\omega(v)$ is the measure on $\mathbb{S}^{d-1}$. 
We also define the spaces $H^k(0,T;L^2(SM))$ for positive integer $k$ with the norm
$$
\|f\|_{H^k(0,T;L^2(SM))} = \Bigg( \sum_{\alpha=0}^k\|\p_t^\alpha f\|^2_{L^2(SM_T)}\Bigg)^{1/2}.
$$
When $p=\infty$, $L^\infty(SM)$, $L^\infty(SM_T)$ and $L^\infty(\p_\pm SM_T)$ are the standard vector spaces consisting of all functions that are essentially bounded.

We first study the forward problem for the linear transport equation in Section~\ref{sec:forward linear}. 
Equipped with this, we apply  the contraction mapping principle to deduce the unique existence of solution to the nonlinear transport equation in Section~\ref{sec:forward nonlinear}.
\subsection{Forward problem for the linear transport equation}\label{sec:forward linear}
We consider the initial boundary value problem for the linear transport equation with the source $S\equiv S(t,x,v)$:
	\begin{align}\label{EQN: linear RTE}
		\left\{\begin{array}{rcll}
			\p_tf + X f + \sigma f &=& K(f) + S & \hbox{in }SM_T,  \\
			f&=&f_0 & \hbox{on }\{0\}\times SM,\\
			f &=& f_- & \hbox{on }\p_-SM_T,
		\end{array}\right.
	\end{align}
where the scattering operator $K$ on the manifold takes the form
\begin{align}\label{def:collision}
	K(f)(t,x,v) := \int_{S_xM} \mu(x,v',v)f(t,x,v') \,dv'.
\end{align}

We will demonstrate the existence of a solution to \eqref{EQN: linear RTE} by proving that the corresponding integral equation has a solution. To achieve this, we study the following simpler case first.  

\begin{proposition}
	Suppose that $\sigma\in L^\infty(SM)$ and $f_-\in L^\infty(\p_-SM_T)$.
	The solution $f$ of the problem 
	\begin{align}\label{EQN:RTE boundary}
		\left\{\begin{array}{rcll}
			\p_tf + X f + \sigma f &=& 0 & \hbox{in }SM_T,  \\
			f &=&0 & \hbox{on }\{0\}\times SM,\\
			f &=& f_- & \hbox{on }\p_-SM_T 
		\end{array}\right.
	\end{align}
is 
\begin{align}\label{f expression boundary}
	f(t,x,v)= H(t-\tau_-) f_- (t-\tau_-,\gamma_{x,v}(-\tau_-),\dot\gamma_{x,v}(-\tau_-)) e^{-\int^{\tau_-}_0\sigma(\gamma_{x,v}(-s),\dot\gamma_{x,v}(-s))ds},
\end{align} 
where $H$ is the Heaviside function, that is, $H$ satisfies $H(s)=0$ if $s<0$ and $H(s)=1$ if $s>0$.  

\end{proposition}
To simplify the notation, in the formulation above we denote $\tau_-:=\tau_-(x,v)$ for a fixed $(x,v)\in SM$.
\begin{proof}
	For a fixed $(x,v)\in SM$ and $0<t<T$, let 
	$$
	F(s):=f(s+t-\tau_-(x,v), \phi_{s-\tau_-(x,v)}(x,v)),\quad \Sigma(s):=\sigma(\phi_{s-\tau_-(x,v)}(x,v)).
	$$
	The equation \eqref{EQN:RTE boundary} can be written as
	$$
	 {d F\over ds}(s) + \Sigma(s)F(s)=0,
	$$
	whose solution is 
	$$
	F(s)=F(0)e^{-\int^{s}_0 \Sigma(\eta)d\eta}.
	$$
	Choosing $s=\tau_-(x,v)$, we have
	$$
	F(\tau_-(x,v))=F(0)e^{-\int^{\tau_-(x,v)}_0 \Sigma(\eta)d\eta},
	$$
	which leads to
	$$
	f(t,x,v) = F(0)e^{-\int^{\tau_-(x,v)}_0 \sigma(\phi_{-\tilde\eta}(x,v))d\tilde\eta}.
	$$
	by applying the change of variable $\tilde{\eta}=-\eta+\tau_-(x,v)$.
	By taking $F(0)=f(t-\tau_-(x,v),\phi_{-\tau_-(x,v)}(x,v))$ which vanishes if $t\leq \tau_-(x,v)$, we obtain the desired result. 
\end{proof}

Let's study the integral formulation of the linear transport equation \eqref{EQN: linear RTE}. 
\begin{proposition}\label{prop:charateristic}
	Suppose that $\sigma\in L^\infty(SM)$ and $\mu\in L^\infty(SM^2)$ satisfy \eqref{EST:sigma} and \eqref{EST:k}. Let  $S\in L^\infty(SM_T)$, $f_0\in L^\infty(SM)$, and $f_-\in L^\infty(\p_-SM_T)$. Then the solution $f$ to \eqref{EQN: linear RTE} satisfies the integral formulation of the transport equation: 
	\begin{align}\label{ID:integral formulation}
	    f(t,x,v)  
		 = \,& f_0(\gamma_{x,v}(-t),\dot\gamma_{x,v}(-t)) e^{-\int^t_0 \sigma(\gamma_{x,v}(-s),\dot\gamma_{x,v}(-s))ds} H(\tau_- -t) \notag\\
		 &+   H(t-\tau_-) f_- (t-\tau_-,\gamma_{x,v}(-\tau_-),\dot\gamma_{x,v}(-\tau_-)) e^{-\int^{\tau_-}_0\sigma(\gamma_{x,v}(-s),\dot\gamma_{x,v}(-s))ds} \notag\\
		 &+  \int^t_0  e^{-\int^s_0 \sigma(\gamma_{x,v}(-r),\dot\gamma_{x,v}(-r))dr} \LC K(f)+S\RC(t-s,\gamma_{x,v}(-s),\dot\gamma_{x,v}(-s))H(\tau_--s)\,ds.
	\end{align}
\end{proposition}
In the Euclidean case, this result can be found in Proposition~4 (page 233), combining with Remark 12, in \cite{Lion6}. To make the paper self contained, we provide below the proof for the Riemannian case.

\begin{proof}
We first consider the homogeneous boundary condition, that is, $f_-=0$.
Multiplying $$e^{\int^{t}_0\sigma(\phi_{\eta+k}(x,v))d\eta}$$ on both sides of the transport equation in \eqref{EQN: linear RTE}, we get
\begin{align}\label{EQN:ODE}
	{d\over dt} \LC e^{\int^{t}_0\sigma(\phi_{\eta+k}(x,v))d\eta}f(t,\phi_{t+k}(x,v))\RC = e^{\int^{t}_0\sigma(\phi_{\eta+k}x,v))d\eta}g(t,\phi_{t+k}(x,v)),
\end{align}
where we denote $g:=K(f)+S$.
By solving the differential equation \eqref{EQN:ODE} and then multiplying $e^{-\int^{t}_0\sigma(\phi_{\eta+k}(x,v))d\eta}$ on both sides of the solution, we have
\begin{align}\label{EQN:ODE 1}
	f(t,\phi_{t+k}(x,v)) &= e^{-\int^{t}_0\sigma(\phi_{\eta+k}(x,v))d\eta}f_0(\phi_k(x,v)) \notag\\
	&\quad+ e^{-\int^{t}_0\sigma(\phi_{\eta+k}(x,v))d\eta}\int^{t}_0 e^{\int^{s}_0\sigma(\phi_{\eta+k}(x,v))d\eta} g(s,\phi_{s+k}(x,v))ds. 
\end{align}
Replacing $\phi_{t+k}(x,v)$ by $\phi_0(x,v)=(x,v)$ (that is, taking $k=-t$) in \eqref{EQN:ODE 1} gives
\begin{align}\label{EQN:ODE 2}
	f(t,x,v) = e^{-\int^{t}_0\sigma( \phi_{\eta-t}(x,v))d\eta}f_0(\phi_{-t}(x,v)) + e^{-\int^{t}_0\sigma(\phi_{\eta-t}(x,v))d\eta}\int^{t}_0 e^{\int^{s}_0\sigma(\phi_{\eta-t}(x,v))d\eta} g(s,\phi_{s-t}(x,v))ds. 
\end{align}
Moreover, we apply the change of variables $\tilde{\eta}=-\eta+t$ so that \eqref{EQN:ODE 2} becomes
\begin{align}\label{EQN:ODE 3}
	f(t,x,v) = e^{-\int^{t}_0\sigma( \phi_{-\tilde\eta}(x,v))d\tilde\eta}f_0(\phi_{-t}(x,v)) + e^{-\int^{t}_0\sigma(\phi_{-\tilde\eta}(x,v))d\tilde\eta}\int^{t}_{0} e^{\int^{t}_{t-s}\sigma(\phi_{-\tilde\eta}(x,v))d\tilde\eta} g(s,\phi_{s-t}(x,v))ds. 
\end{align}
We then apply another change of variables $\tilde{s}=-s+t$ so that
\begin{align}\label{EQN:ODE 4}
	\int^{t}_{0} e^{\int^{t}_{t-s}\sigma(\phi_{-\tilde\eta}(x,v))d\tilde\eta} g(s,\phi_{s-t}(x,v))ds = 
	\int^{t}_{0} e^{-\int_{t}^{\tilde{s}}\sigma(\phi_{-\tilde\eta}(x,v))d\tilde\eta} g(t-\tilde{s},\phi_{-\tilde{s}}(x,v))d\tilde{s}. 
\end{align}
From \eqref{EQN:ODE 3} and \eqref{EQN:ODE 4}, taking $f_0(\phi_{-t}(x,v))=0$ if $\phi_{-t}(x,v)\notin \Omega$ (namely, $t\geq\tau_-(x,v)$), we derive that the solution satisfies the integral equation with $f_-\equiv 0$.

Next, in the case of a nonhomogeneous boundary condition $f_-\neq 0$, we let $f_1$ be the solution of \eqref{EQN:RTE boundary} and look for the solution $f$ of the problem \eqref{EQN: linear RTE} in the form $f=f_1+w$, where 
$w$ is the solution of 
	\begin{align}\label{EQN: linear RTE w}
	\left\{\begin{array}{rcll}
		\p_tw + X w + \sigma w &=& K(f_1+w) + S & \hbox{in }SM_T,  \\
		w &=&f_0 & \hbox{on }\{0\}\times SM,\\
		w &=& 0 & \hbox{on }\p_-SM_T.
	\end{array}\right.
\end{align}
Since $w$ has the homogeneous boundary condition, $w$ satisfies the integral equation with $f_-=0$. Therefore, combining this with \eqref{f expression boundary}, we finally deduce that $f=f_1+w$ satisfies \eqref{ID:integral formulation}.
\end{proof}

In the following we will see that solving the integral equation \eqref{ID:integral formulation} is equivalent to solving \eqref{EQN: linear RTE}.
Hence once we show that the integral equation \eqref{ID:integral formulation} has a unique solution, this is sufficient to say that the well-posedness of \eqref{EQN: linear RTE} holds.
\begin{proposition}\label{pro:integral equation}
	Under the hypothesis of Proposition~\ref{prop:charateristic}, if $f$ satisfies the integral equation \eqref{ID:integral formulation}, then $f$ is the solution to \eqref{EQN: linear RTE}. 	
	Moreover, there exists a unique solution to the integral equation \eqref{ID:integral formulation}. 
\end{proposition}
\begin{proof} 
\noindent\textit{Step 1: Equivalence.} Below we will show that if there exists a function $f$ satisfying \eqref{ID:integral formulation}, then such $f$ is a solution to \eqref{EQN: linear RTE}. Notice that  
\begin{align*}
    (\p_t+X)f(t,x,v)=\frac{d}{dk} f(t+k,\phi_k(x,v))|_{k=0}.
\end{align*}
We apply the operator $\p_t+X$ to the right-hand side of the integral formula \eqref{ID:integral formulation} to get
\begin{align*}
   &\quad  (\p_t+X) f(t,x,v) \\
   & = \frac{d}{dk}\bigg\{f_0(\phi_{-(t+k)}(\phi_k(x,v))e^{-\int_0^{t+k}\sigma(\phi_{-s}(\phi_k(x,v)))ds} H(\tau_-(\phi_k(x,v))-t-k)\\
    & \quad+H(t+k-\tau_-(\phi_k(x,v)))f_-(t+k-\tau_-(\phi_k(x,v)),\phi_{-\tau_-(\phi_k(x,v))}(\phi_k(x,v))) e^{-\int_0^{\tau_-(\phi_k(x,v))}\sigma(\phi_{-s}(\phi_k(x,v)))ds}\\
    &\quad +\int_0^{t+k} e^{-\int_0^s\sigma(\phi_{-r}(\phi_k(x,v)))dr}(K(f)+S)(t+k-s,\phi_{-s}(\phi_k(x,v)))H(\tau_-(\phi_k(x,v))-s)\,ds \bigg\}\bigg|_{k=0}\\
    &=\frac{d}{dk}\bigg\{f_0(\phi_{-t}(x,v))e^{-\int_0^{t+k}\sigma(\phi_{-s+k}(x,v))ds} H(\tau_-(x,v)-t)\\
    &\quad  +H(t-\tau_-(x,v))f_-(t-\tau_-(x,v),\phi_{-\tau_-(x,v)}(x,v)) e^{-\int_0^{\tau_-(x,v)+k}\sigma(\phi_{-s+k}(x,v))ds}\\
    &\quad+\int_0^{t+k} e^{-\int_0^s\sigma(\phi_{-r+k}(x,v))dr}(K(f)+S)(t+k-s,\phi_{-s+k}(x,v))H(\tau_-(x,v)+k-s)\,ds
    \bigg\}\bigg|_{k=0}\\
    &=:I_1+I_2+I_3.
\end{align*}
Here we used the fact that $\tau_-(\phi_k(x,v))=\tau_-(x,v)+k$.

Now we consider $I_1$ - $I_3$ separately. For $I_1$, we have
\begin{align*}
    I_1&= \frac{d}{dk}\LC f_0(\phi_{-t}(x,v))e^{-\int_0^{t+k}\sigma(\phi_{-s+k}(x,v))ds} H(\tau_-(x,v)-t)\RC \bigg|_{k=0}\\
    &=  f_0(\phi_{-t}(x,v))e^{-\int_0^{t}\sigma(\phi_{-s}(x,v))ds}\left(-\sigma(\phi_{-t}(x,v))-\int_0^{t} X\sigma(\phi_{-s}(x,v))ds \right) H(\tau_-(x,v)-t)\\
    &= -f_0(\phi_{-t}(x,v))e^{-\int_0^{t}\sigma(\phi_{-s}(x,v))ds}\sigma(x,v) H(\tau_-(x,v)-t).
\end{align*}
For $I_2$, 
\begin{align*}
    I_2&=  \frac{d}{dk}\LC H(t-\tau_-(x,v))f_-(t-\tau_-(x,v),\phi_{-\tau_-(x,v)}(x,v)) e^{-\int_0^{\tau_-(x,v)+k}\sigma(\phi_{-s+k}(x,v))ds}\RC \bigg|_{k=0}\\
    &=  H(t-\tau_-(x,v))f_-(t-\tau_-(x,v),\phi_{-\tau_-(x,v)}(x,v))\\ & \quad\quad e^{-\int_0^{\tau_-(x,v)}\sigma(\phi_{-s}(x,v))ds}\big(-\sigma(\phi_{-\tau_-(x,v)}(x,v))-\int_0^{\tau_-(x,v)} X\sigma (\phi_{-s}(x,v))ds  \big)\\
    &=  -H(t-\tau_-(x,v))f_-(t-\tau_-(x,v),\phi_{-\tau_-(x,v)}(x,v)) e^{-\int_0^{\tau_-(x,v)}\sigma(\phi_{-s}(x,v))ds}\sigma(x,v).
\end{align*}
We denote $m=s-k$, then
\begin{align*}
    I_3&=  \frac{d}{dk} \LC\int_0^{t+k} e^{-\int_0^s\sigma(\phi_{-r+k}(x,v))dr}(K(f)+S)(t+k-s,\phi_{-s+k}(x,v))H(\tau_-(x,v)+k-s)\,ds \RC \bigg|_{k=0}\\
    &= \frac{d}{dk} \LC \int_{-k}^t e^{-\int_{-k}^m \sigma(\phi_{-\nu}(x,v))d\nu} (Kf+S)(t-m,\phi_{-m}(x,v)) H(\tau_-(x,v)-m)\,dm\RC \bigg|_{k=0}\\
    &=  (Kf+S)(t,x,v)\\
    & \quad+\int_0^t e^{-\int_{0}^m \sigma(\phi_{-\nu}(x,v))d\nu}\big(-\sigma(x,v)\big) (Kf+S)(t-m,\phi_{-m}(x,v)) H(\tau_-(x,v)-m)\,dm.
\end{align*}
Combining the above 3 terms together, we have
\begin{align*}
 (\p_t+  X) f (t,x,v) =I_1+I_2+I_3=-\sigma(x,v) f+(Kf+S)(t,x,v).
\end{align*}
Finally, it's easy to check that $f(0,x,v)=f_0(x,v)$ if $(x,v)\in SM$, and $f(t,x,v)=f_-(t,x,v)$ if $(x,v)\in \p_-SM$ and $t>0$. We thus conclude that $f$ is a solution to \eqref{EQN: linear RTE}. Combining with Proposition~\ref{prop:charateristic}, we see that to show the forward problem of \eqref{EQN: linear RTE}, it is sufficient to find a solution to the integral equation. 

\noindent\textit{Step 2: Existence of solutions to the integral equation.} 
We define a sequence of functions $f^{(n)}$ in the following ways:
\begin{align}\label{EST:f0}
	f^{(0)}(t,x,v) &= f_0(\phi_{-t}(x,v)) e^{-\int^t_0 \sigma(\phi_{-s}(x,v))ds} H(\tau_-(x,v)-t) \notag\\
	&\quad + f_- (t-\tau_-(x,v),\phi_{-\tau_-(x,v)}(x,v)) e^{-\int^{\tau_-(x,v)}_0\sigma(\phi_{-s}(x,v))ds} H(t-\tau_-(x,v))\notag\\
	&\quad + \int^t_0 e^{-\int^s_0 \sigma(\phi_{-r}(x,v))dr} S (t-s,\phi_{-s}(x,v))H(\tau_-(x,v)-s)\,ds
\end{align}
and for $n\geq 0$,
\begin{align}\label{EST:f n+1}
	f^{(n+1)}(t,x,v) 
	= f^{(0)}(t,x,v)  
	+\int^t_0 e^{-\int^s_0 \sigma(\phi_{-r}(x,v))dr} K(f^{(n)}) (t-s,\phi_{-s}(x,v))H(\tau_-(x,v)-s)\,ds.
\end{align}
Let $w^{(n+1)}:= f^{(n+1)}-f^{(n)}$ for $n\geq 0$ and then be represented as 
\begin{align*}
	w^{(n+1)}(t,x,v)
	= \int^t_0 e^{-\int^s_0 \sigma(\phi_{-r}(x,v))dr} K(w^{(n)}) (t-s,\phi_{-s}(x,v))H(\tau_-(x,v)-s)\,ds. 
\end{align*}
Recall that in \eqref{EST:k} for almost every $(x,v)\in SM$, $\mu$ satisfies
$$
\int_{S_xM} \mu(x,v',v)\,dv'\leq \sigma(x,v).
$$
From this, we can derive that
\begin{align}\label{EST:w n+1}
	\LV w^{(n+1)}(t,x,v)\RV
	&\leq \LC\int^t_0 e^{-\int^s_0 \sigma(\phi_{-r}(x,v))dr} \sigma (\phi_{-s}(x,v))H(\tau_-(x,v)-s)\,ds \RC \|w^{(n)}\|_{L^\infty(SM_T)} \notag\\
	&=\left\{\begin{array}{ll}
		\LC\int^t_0 e^{-\int^s_0 \sigma(\phi_{-r}(x,v))dr} \sigma (\phi_{-s}(x,v)) \,ds \RC \|w^{(n)}\|_{L^\infty(SM_T)} &\quad \hbox{ if } t<\tau_-(x,v);\\
		\LC\int^{\tau_-(x,v)}_0 e^{-\int^s_0 \sigma(\phi_{-r}(x,v))dr} \sigma (\phi_{-s}(x,v))\,ds \RC \|w^{(n)}\|_{L^\infty(SM_T)} &\quad \hbox{ if } t>\tau_-(x,v);\\
	\end{array}
	\right. \notag \\
	&=\left\{\begin{array}{ll}
		\LC 1- e^{-\int^t_0\sigma(\phi_{-r}(x,v))dr}\RC\|w^{(n)}\|_{L^\infty(SM_T)} &\quad \hbox{ if } t<\tau_-(x,v);\\
		\LC 1- e^{-\int^{\tau_-}_0\sigma(\phi_{-r}(x,v))dr}\RC\|w^{(n)}\|_{L^\infty(SM_T)} &\quad \hbox{ if } t>\tau_-(x,v);
	\end{array}
	\right.
\end{align}
for $(t,x,v)\in SM_T$. We then denote the scalar value $\kappa$ by
$$
\kappa:= \sup_{(x,v)\in SM} \LC 1- e^{-\int^{\tau_-(x,v)}}_0\sigma(\phi_{-r}(x,v))dr\RC.
$$
It is clear that $0\leq\kappa<1$ since $0\leq \sigma\leq \sigma^0$. Due to the monotonicity of $e^{-\int^s_0\sigma(\phi_{-r}(x,v))dr}$ with respect to $s$, we obtain
\begin{align}\label{EST:w n 0}
	\|w^{(n+1)}\|_{L^\infty(SM_T)}\leq \kappa \|w^{(n)}\|_{L^\infty(SM_T)}\leq \kappa^n \|w^{(1)}\|_{L^\infty(SM_T)}\leq \kappa^{n+1} \|f^{(0)}\|_{L^\infty(SM_T)}.
\end{align}

Next, we estimate the third term on the right-hand side of \eqref{EST:f0}. From \eqref{EST:sigma}, we derive that 
\begin{align*}
	&\LV \int^t_0 e^{-\int^s_0 \sigma(\phi_{-r}(x,v))dr} S (t-s,\phi_{-s}(x,v))H(\tau_-(x,v)-s)\,ds\RV\\
	\leq\, & \|S\|_{L^\infty(SM_T)} \LC \int^{T}_0  e^{-\int^s_0 \sigma(\phi_{-r}(x,v))dr}H(\tau_-(x,v)-s) \,ds \RC\\
	\leq\, & T \|S\|_{L^\infty(SM_T)}.
\end{align*} 
Thus \eqref{EST:f0} and $\sigma\geq 0$ lead to 
\begin{align}\label{EST: f0}
	\|f^{(0)}\|_{L^\infty(SM_T)} \leq 
	\|f_0\|_{L^\infty(SM)} + \|f_-\|_{L^\infty(\p_-SM_T)} +  T \|S\|_{L^\infty(SM_T)}.
\end{align}

Combining these estimates \eqref{EST:w n 0}-\eqref{EST: f0} together, we can derive that
\begin{align}\label{EST: w0 1}
	\|w^{(n+1)}\|_{L^\infty(SM_T)} \leq \kappa^{n+1} \LC \|f_0\|_{L^\infty(SM)} +\|f_-\|_{L^\infty(\p_-SM_T)} +   T \|S\|_{L^\infty(SM_T)}\RC 
\end{align}
with $0\leq  \kappa<1$. 
This implies that the series $\sum^\infty_{n=0}w^{(n+1)}$ is convergent and thus the partial sum
$$
f^{(0)} + \sum^n_{k=0} w^{(k+1)} = f^{(n+1)} 
$$
converges to a limit $f$ in $L^\infty(SM_T)$. In particular, $f$ satisfies the integral equation:
\begin{align*}
	f (t,x,v) 
	= f^{(0)}(t,x,v)  
	+\int^t_0 e^{-\int^s_0 \sigma(\phi_{-r}(x,v))dr} K(f) (t-s,\phi_{-s}(x,v))H(\tau_-(x,v)-s)\,ds 
\end{align*}
and, furthermore, $f$ is also a solution of \eqref{EQN: linear RTE} due to Step 1.

\noindent\textit{Step 3: Unique solution for the integral equation.}  Finally we show the uniqueness of the solution.
Let $f_1$ and $f_2$ in $L^\infty(SM_T)$ be the solutions to \eqref{ID:integral formulation}. Let $w:=f_1-f_2\in L^\infty(SM_T)$. Then $w$ satisfies the integral equation:
\begin{align*}
	& w(t,x,v)=\int^t_0 e^{-\int^s_0 \sigma(\phi_{-r}(x,v))dr} K(w) (t-s,\phi_{-s}(x,v))H(\tau_-(x,v)-s)\,ds.
\end{align*}
Following the argument as in \eqref{EST:w n+1}, we obtain
\begin{align*}
	\|w\|_{L^\infty(SM_T)} \leq \kappa \|w\|_{L^\infty(SM_T)},\quad 0\leq\kappa<1.
\end{align*}
This implies that $w\equiv 0$. 
\end{proof}

From the above discussion, we have shown that there exists a unique solution $f$ to the integral equation. Due to the equivalence, such $f$ is also a solution to \eqref{EQN: linear RTE}.
Hence we can now conclude the following well-posedness result for the problem \eqref{EQN: linear RTE}.
\begin{proposition}[Well-posedness for linear transport equation]\label{prop:forward problem linear}
	Suppose that $\sigma\in L^\infty(SM)$ and $\mu\in L^\infty(SM^2)$ satisfy \eqref{EST:sigma} and \eqref{EST:k}. Let  $S\in L^\infty(SM_T)$, $f_0\in L^\infty(SM)$ and $f_-\in L^\infty(\p_-SM_T)$.
	We consider the following problem:
	\begin{align}\label{EQN: linear RTE source}
	\left\{\begin{array}{rcll}
		\p_tf + X f + \sigma f &=& K(f) + S & \hbox{in }SM_T,  \\
		f &=&f_0 & \hbox{on }\{0\}\times SM,\\
		f &=& f_- & \hbox{on }\p_-SM_T.
	\end{array}\right.
\end{align}
Then \eqref{EQN: linear RTE source} has a unique solution $f$ in $L^\infty(SM_T)$ satisfying
\begin{align}\label{stability}
	\|f\|_{L^\infty(SM_T)}\leq C\LC\|f_0\|_{L^\infty(SM)} + \|f_-\|_{L^\infty(\p_-SM_T)} + \|S\|_{L^\infty(SM_T)}\RC,
\end{align}
where the constant $C$ depends on $\sigma$, $T$.
\end{proposition}

\begin{proof}
From Proposition~\ref{prop:charateristic} and Proposition~\ref{pro:integral equation}, it is clear that the solution $f$ to \eqref{EQN: linear RTE source} uniquely exists. Moreover, using a similar argument as in \eqref{EST:w n+1}, we can derive the stability estimate \eqref{stability} from \eqref{ID:integral formulation}.
\end{proof}

It has been proved in [\cite{Lion6}, Theorem~3, p229] that when $f_-\equiv 0$, $S\geq 0$ and $f_0\geq 0$, the solution is nonnegative. 
In the following proposition, we show the maximum principle for the transport equation, namely, the solution to \eqref{EQN: linear RTE source} is strictly positive if $S\geq 0$, the initial and boundary data are strictly positive.
	\begin{proposition}[Maximum principle]\label{prop:positive solution} Suppose the hypotheses in Proposition~\ref{prop:forward problem linear} hold and suppose that $S\geq 0$.  
		If $f_0\geq c> 0$ and $f_-\geq c>0$ for some positive constant $c$, then there exists a positive constant $\tilde{c}$ such that $f\geq \tilde{c} >0$ in $SM_T$.
	\end{proposition} 
	\begin{proof}
		From \eqref{EST:f0}, $\sigma\leq \sigma^0$ in \eqref{EST:sigma}, and the hypothesis $f_0,\,f_-\geq c> 0$, we obtain
		$$
		f^{(0)}(t,x,v)\geq e^{-T\sigma^0}c>0\quad \hbox{almost everywhere (a.e.)}. 
		$$
		This implies $K(f^{(0)})\geq 0$ due to $\mu\geq 0$. 
		Hence, by induction, we can derive from \eqref{EST:f n+1} that for $n\geq 0$,
		$$
		f^{(n+1)}(t,x,v)\geq f^{(n)}(t,x,v)\geq f^{(0)}(t,x,v) \geq e^{-T\sigma^0}c>0\quad a.e..
		$$
		We therefore have an increasing sequence converging to a function $f(t,x,v)$, which satisfies $f(t,x,v)\geq e^{-T\sigma^0}c>0$.
		Alternatively, we can apply the proof in Proposition~\ref{pro:integral equation}, which gives that $f^{(n)}\rightarrow f$ in $L^\infty (SM_T)$ as $n\rightarrow \infty$. Hence this also leads to the same result, that is,
		$$
		f(t,x,v) \geq e^{-T\sigma^0}c>0\quad a.e..
		$$
		This completes the proof.
	\end{proof}

\subsection{Forward problem for the nonlinear transport equation}\label{sec:forward nonlinear}
Equipped with the well-posedness result for the linear equation, we will prove the unique existence of solution for the following problem:
\begin{align}\label{EQN: RTE equ}
	\left\{\begin{array}{rcll}
		\p_tf + X f + \sigma f + N(x,v,f) &=& K(f) & \hbox{in } SM_T,  \\
		f &=&f_0 & \hbox{on }\{0\}\times SM,\\
		f &=&f_- & \hbox{on }\p_-SM_T.
	\end{array}\right.
\end{align}

\begin{theorem}[Well-posedness for nonlinear transport equation]\label{THM:WELL}
	Let $M$ be the interior of a compact non-trapping Riemannian manifold $\overline M$ with strictly convex boundary $\p M$. 
	Suppose that $\sigma$ and $k$ satisfy \eqref{EST:sigma} and \eqref{EST:k}.
	Then there exists a small parameter $0<\delta<1$ such that for any
	\begin{equation}\label{DEF:set X}
	\begin{split}
	(f_0,f_-) \in \mathcal{X}^M_\delta:=  \{(f_0,f_-)\in L^\infty(SM)\times L^\infty(\p_-SM_T):\, \|f_0\|_{L^\infty(SM)}\leq \delta,  \|f_-\|_{L^\infty(\p_-SM_T)}\leq \delta\},
	\end{split}
	\end{equation}
	the problem \eqref{EQN: RTE equ} has a unique small solution $f\in L^\infty(SM_T)$ satisfying
	$$
	\|f\|_{L^\infty(SM_T)}\leq C \LC \|f_0\|_{L^\infty(SM)}+\|f_-\|_{L^\infty(\p_-SM_T)} \RC, 
	$$
	where the positive constant $C$ is independent of $f$, $f_0$ and $f_-$.
\end{theorem}
\begin{proof}
	To show the existence, let $(f_0,f_-) \in \mathcal{X}^M_\delta$, we first consider the following problem for the linear equation:
	\begin{equation}\label{EQN: linear}
		\left\{\begin{array}{rcll}
			\p_t \hat{f}+ X \hat{f} + \sigma \hat{f} &=& K (\hat{f})  &\text{in}\ SM_T, \\
			\hat{f} &=&	f_0 & \hbox{on }\{0\}\times SM,\\
			\hat{f} &= & f_- &\text{on}\ \p_-SM_T\, .
		\end{array}\right.
	\end{equation}
	By Proposition~\ref{prop:forward problem linear}, there exists a unique solution $\hat{f}$ of \eqref{EQN: linear} that satisfies
	\begin{align}\label{estimate f0}
		\|\hat{f}\|_{L^\infty(SM_T)}\leq C\LC\|f_-\|_{L^\infty(\p_-SM_T)}+\|f_0\|_{L^\infty(SM)} \RC\leq 2C\delta,
	\end{align}
	where the constant $C>0$ is independent of $\hat{f}$, $f_-$ and $f_0$.
	
	Now we let $w:=f-\hat{f}$. We observe that if such function $w$ exists, then $w$ must satisfies the following problem:
	\begin{equation}\label{EQN: linear remining}
		\left\{\begin{array}{rcll}
			\p_t w+ X w + \sigma w &=& K (w) -  N(x,v, w+\hat{f}) &\text{in}\ SM_T, \\
			w &=&0& \hbox{on }\{0\}\times SM,\\
			w &= & 0 &\text{on}\ \p_-SM_T\, .
		\end{array}\right.
	\end{equation}
To prove \eqref{EQN: linear remining} has a solution, we apply the contraction mapping principle. We denote the set 
$$
\mathcal{G}:=\{\varphi\in L^\infty(SM_T):\, \varphi|_{t=0}=0,\quad \varphi|_{\p_-SM_T}=0,\quad \hbox{and}\quad\|\varphi\|_{L^\infty(SM_T)}\leq \eta\},
$$ 
where the parameter $\eta>0$ will be determined later. For $\varphi\in \mathcal{G}$, we define the function $F$ by
$$
    F(\varphi):= N(x,v,\varphi+\hat{f}).
$$
Then $F(\varphi)\in L^\infty(SM_T)$ due to \eqref{estimate f0} and the hypothesis of $N(f)$.
In particular, Proposition~\ref{prop:forward problem linear} yields that the problem
\begin{equation}\label{EQN: linear remining 1}
\left\{\begin{array}{rcll}
	 \p_t \tilde{w}+ X \tilde{w} + \sigma \tilde{w} &=& K (\tilde{w}) - F(\varphi) &\text{in}\ SM_T , \\
	\tilde{w} &=&0& \hbox{on }\{0\}\times SM,\\
	\tilde{w} &= & 0 &\text{on}\ \p_-SM_T ,
\end{array}\right.
\end{equation}
is uniquely solvable for any $\varphi\in \mathcal{G}$. We now denote $\mathcal{L}^{-1}:F(\varphi)\in L^\infty(SM_T)\mapsto \tilde{w}\in L^\infty(SM_T)$ the solution operator for the problem \eqref{EQN: linear remining 1} and also define the map $\Psi$ on the set $\mathcal{G}$ by 
$$
\Psi(\varphi) : = (\mathcal{L}^{-1} \circ F)(\varphi).
$$
In the following, we will show that $\Psi$ is a contraction map on $\mathcal{G}$.
To this end, we first show that $\Psi(\mathcal{G})\subset \mathcal{G}$. Taking $\varphi\in \mathcal{G}$, from \eqref{condition N}, the Taylor's Theorem, and Proposition~\ref{prop:forward problem linear}, we derive that
\begin{align*}
	\|\Psi(\varphi)\|_{L^\infty(SM_T)} 
	&=\|\mathcal{L}^{-1} (F(\varphi))\|_{L^\infty(SM_T)}\leq C \|F(\varphi)\|_{L^\infty(SM_T)}\\
	&= C \|N(x,v,\varphi+\hat{f})\|_{L^\infty(SM_T)}\\
	&\leq C\| \p^2_zN(x,v,0) (\varphi+\hat{f})^2+ N_r(x,v,\varphi+\hat{f})(\varphi+\hat{f})^3\|_{L^\infty(SM_T)}\\
	&\leq C\LC(\delta+\eta)^2+(\delta+\eta)^3\RC,
\end{align*}
where constant $C>0$ is independent of $\delta$ and $\eta$. Note that both $\p^2_zN(x,v,0)$ and
$$
N_r(x,v,\varphi+\hat{f}):=\int^1_0(1-s)^2\p_z^3N(x,v,s(\varphi+\hat{f}))ds
$$
are bounded in $SM_T$. We then take $\delta,\, \eta$ sufficiently small with $0<\delta<\eta<1$ such that
$$
C\LC(\delta+\eta)^2+(\delta+\eta)^3\RC < \eta,
$$
which implies $\Psi$ maps $\mathcal{G}$ into itself.

Moreover, for any $\varphi_1,\, \varphi_2 \in \mathcal{G}$, from Proposition \ref{prop:forward problem linear}, we can also derive that
\begin{align*}
	\|\Psi(\varphi_1)-\Psi(\varphi_2)\|_{L^\infty(SM_T)} 
	&=\|\mathcal{L}^{-1} (F(\varphi_1))-\mathcal{L}^{-1} (F(\varphi_2))\|_{L^\infty(SM_T)}\\
	&\leq C\|F(\varphi_1)-F(\varphi_2)\|_{L^\infty(SM_T)}.
\end{align*}
We estimate 
\begin{align*}
    &\quad \|N(x,v,\varphi_1+\hat{f})-N(x,v,\varphi_2+\hat{f})\|_{L^\infty(SM_T)}  \\
    &\leq C\|\p_z^2N(x,v,0)((\varphi_1+\hat{f})^2-(\varphi_2+\hat{f})^2)\|_{L^\infty(SM_T)} \\
    &\quad + C\| N_r(x,v,\varphi_1+\hat{f}) ((\varphi_1+\hat{f})^3-(\varphi_2+\hat{f})^3))\|_{L^\infty(SM_T)}\\
    &\quad + C\|(N_r(x,v,\varphi_1+\hat{f}) - N_r(x,v,\varphi_2+\hat{f})) (\varphi_2+\hat{f})^3\|_{L^\infty(SM_T)}\\
	&\leq C\LC (\delta+\eta) +(\delta+\eta)^2 +(\delta+\eta)^3\RC\| \varphi_1 - \varphi_2\|_{L^\infty(SM_T)}.
\end{align*}
Here we used the fact that $N_r$ is Lipschitz in $z$ with the Lipschitz constant independent of $x,\, v$ due to the boundedness of $\p_z^kN$.
In addition, we choose small $\delta,\, \eta$ so that  
$$
    C\LC (\delta+\eta) +(\delta+\eta)^2 +(\delta+\eta)^3\RC < 1.
$$
This yields that $\Psi$ is a contraction map.
By the contraction mapping principle, there exists a unique $w\in\mathcal{G}$ so that 
$\Psi(w)=w$, which then satisfies the problem \eqref{EQN: linear remining}. Also $w$ satisfies the estimate  
\begin{align}
    \|w\|_{L^\infty(SM_T)} &=\|\Psi(w)\|_{L^\infty(SM_T)} \leq C \LC(\delta+\eta)  +(\delta+\eta)^{2}\RC\LC \| w\|_{L^\infty(SM_T)}+\|\hat{f}\|_{L^\infty(SM_T)}\RC \notag.
\end{align}
We further take $\delta,\,\eta$ small enough so that $C \LC(\delta+\eta)  +(\delta+\eta)^{2}\RC\leq 1/2$ and, therefore, the term containing $\| w\|_{L^\infty(SM_T)}$ on the right-hand side can then be absorbed by the left-hand side, it follows that 
$$\|w\|_{L^\infty(SM_T)}\leq \|\hat{f}\|_{L^\infty(SM_T)}.$$
Finally we conclude that $f=w+\hat{f}$ is the solution to the problem \eqref{EQN: RTE equ} and it satisfies
\begin{align*}
\|f\|_{L^\infty(SM_T)}&\leq \|w\|_{L^\infty(SM_T)}+\|\hat{f}\|_{L^\infty(SM_T)} \\
&\leq   2\|\hat{f}\|_{L^\infty(SM_T)} \notag\\
&\leq  C \LC \|f_0\|_{L^\infty(SM)}+\|f_-\|_{L^\infty(\p_-SM_T)} \RC
\end{align*}
due to \eqref{estimate f0}. This completes the proof.
\end{proof}

\section{Inverse problems in the Euclidean space}\label{sec:Euclidean}
In this section, we will discuss the inverse problem for the nonlinear transport equation in the Euclidean space. The main objective is to show that the nonlinear term as well as the absorption coefficient (or scattering coefficient) can be recovered from the boundary measurements. Notice that as mentioned previously, the well-posedness result in Section~\ref{sec:Preliminaries} also holds in the domain $\Omega$ in $\R^d$.

Recall the following notations in Section~\ref{sec:intro}:
$$
S\Omega:= \Omega\times\mathbb{S}^{d-1},\quad S\Omega^2:= \Omega\times\mathbb{S}^{d-1}\times\mathbb{S}^{d-1},\quad \hbox{and}\quad S\Omega_T:=(0,T)\times \Omega\times\mathbb{S}^{d-1}\quad \hbox{for $T>0$}.
$$ 
Suppose that the absorption coefficient $\sigma\in L^\infty(S\Omega)$ and scattering coefficient $\mu\in L^\infty(S\Omega^2)$ are known and satisfy \eqref{EST:sigma} and \eqref{EST:k}. We consider the nonlinear term $N$ that satisfies \eqref{condition N} and takes the form
$$
N(x,v,z)= \sum_{k=2}^\infty q^{(k)}(x,v){z^k\over k!}, 
$$	
where $q^{(k)}(x,v) = \p_{z}^kN(x,v,0)\in L^\infty(S\Omega)$ and the series converges in $L^\infty(S\Omega)$.  

Let $f$ be the solution to the initial boundary value problem:
\begin{align}\label{EQN: RTE equ Rn}
	\left\{\begin{array}{rcll}
		\p_tf + v\cdot \nabla_x f + \sigma f + N(x,v, f) &=& K(f) & \hbox{in } S\Omega_T ,  \\
		f &=& f_0 & \hbox{on } \{0\}\times S\Omega,\\
		f &=&f_- & \hbox{on }\p_-S\Omega_T.
	\end{array}\right.
\end{align}
The unique existence of small solution $f$ follows by applying Theorem~\ref{THM:WELL}, which is also valid in the Euclidean space.
Recall that we denote the measurement operator by 
\begin{align}\label{bdry operator}
	\mathcal{A}_{\sigma,\mu,N}: (f_0,f_-)\in L^\infty(S\Omega)\times L^\infty(\p_-S\Omega_T)\mapsto f|_{\p_+S\Omega_T}\in L^\infty(\p_+S\Omega_T).
\end{align}

In Section~\ref{sec:Preliminaries}, we have defined backward/forward exit time in the Riemannian manifold. We will adapt these definitions in the Euclidean setting here. 
For $(x,v)\in S\Omega$, the backward exit time $\tau_-(x,v)$ is defined by 
$$
\tau_-(x,v) := \sup\{s> 0:\, x-\eta v\in\Omega \hbox{ for all }0<\eta<s\}.
$$
This is the time at which a particle $x\in \Omega$ with velocity $-v$ leaves the domain $\Omega$. 
Similarly, we define the forward exit time $\tau_+(x,v)$ for every $(x,v)\in S\Omega$ by
$$
\tau_+(x,v) := \sup\{s>0:\, x+\eta v\in\Omega \hbox{ for all }0<\eta<s\}.
$$
In particular, when $(x,v)\in \p_\pm S\Omega$, we have $\tau_\pm(x,v)=0$.  
Suppose that $T$ is sufficiently large so that $T>\text{diam}\,\Omega$, where the notation $\text{diam}\,\Omega$ denotes the diameter of $\Omega$.

This section is structured as follows. We first study the reconstruction of the linear coefficients in Section~\ref{sec:sigma mu} under suitable assumptions. Standing on this result, we will show that the nonlinear term can be uniquely determined from the measurement in Section~\ref{sec:taylor}.

\subsection{Recover $\sigma$ or $\mu$}\label{sec:sigma mu}
To recover the unknown $\sigma$ and $\mu$, we apply the first order linearization to reduce the nonlinear equation to a linear equation without the unknown $N(x,v,f)$. From this, the Carleman estimate for the transport equation is applied to achieve the goal.

For small parameter $\varepsilon$, the well-posedness result in Theorem~\ref{THM:WELL} yields that there is a unique small solution $f(t,x,v)\equiv f(t,x,v;\varepsilon)$ to \eqref{EQN: RTE equ Rn} with initial data $f|_{t=0}=\varepsilon h$ and boundary data $f|_{\p_-S\Omega_T} =\varepsilon g$.
We can obtain the differentiability of the solution $f=f(t,x,v;\varepsilon)$ with respect to $\varepsilon$ by adapting the proof of [\cite{LaReZh21}, Proposition~A.4], where the differentiability is discussed for a nonlinear transport equation, to our setting.   
Hence, we have the $k$-th derivative of $f$ with respect to $\varepsilon$ at $\varepsilon=0$, which is defined by 
$$
F^{(k)}(t,x,v) : =\partial^k_\varepsilon|_{\varepsilon=0} f(t,x,v;\varepsilon) 
$$
for any integer $k\geq 1$.

Now we perform the first linearization of the problem \eqref{EQN: RTE equ Rn} with respect to $\varepsilon$ at $\varepsilon=0$. 
Due to the well-posedness result, the nonlinear term is eliminated and only the linear terms are preserved. Then \eqref{EQN: RTE equ Rn} becomes
\begin{align}\label{EQN: linear RTE equ IP}
	\left\{\begin{array}{rcll}
		\p_t F^{(1)} + v\cdot \nabla_x F^{(1)}+ \sigma F^{(1)}&=& K(F^{(1)}) & \hbox{in } S\Omega_T,  \\
		F^{(1)} &=& h & \hbox{on }\{0\}\times S\Omega,\\
		F^{(1)} &=&g & \hbox{on }\p_-S\Omega_T.
	\end{array}\right.
\end{align}

Hence the problem is reduced to studying the inverse coefficient problem for the above linear transport equation.
Note that the unique determination of $(\sigma, \mu)$ from the albedo operator was shown in \cite{CS2, CS3, CS98} by applying the singular decomposition of the operator under suitable assumptions. 
One might recover both $\sigma$ and $\mu$ by directly applying these existing results for the linear equation. However, additional assumptions might be needed to deduce the uniqueness and stability results in our setting. Therefore, to be consistent with the assumptions we have made in this paper, we will only focus on applying the Carleman estimate to recover either $\sigma$ or $\mu$ by assuming that the other one is given.

Let's briefly discuss how to build the Carleman estimate for the transport equation with linear Carleman weight function $\varphi$, see also \cite{Machida14}. First we note that 
the Carleman estimate is valid under the geometric assumption on the velocity. 
For a fixed vector $\gamma\in\mathbb{S}^{d-1}$, we denote the subset $V$ of the unit sphere by
$$
    V:=\{v\in\mathbb{S}^{d-1}:\, \gamma\cdot v\geq \gamma_0>0\}
$$ 
for some positive constant $\gamma_0$. 
For a fixed $0<\beta<\gamma_0$, there exists a constant $a>0$ so that $\gamma\cdot v-\beta \geq a>0$ in V. Then we define the function 
$$
    B(v):=\gamma\cdot v-\beta. 
$$
Next we define the weight function $\varphi \in C^2([0,T]\times \overline\Omega)$ by
\begin{align}\label{phase}
	\varphi(t,x)= \gamma\cdot x-\beta t. 
\end{align} 
It follows that $(\p_t+v\cdot\nabla_x)\varphi = B(v)>0$, which is essential in the derivation of the Carleman estimate later.

Moreover, we define the transport operator 
$$Pf :=\p_tf + v\cdot \nabla_x f + \sigma f.$$
Let $w(t,x,v)=e^{s\varphi}f(t,x,v)$ for $s>0$. We define the linear operator $L$ by
\begin{align*}
	Lw := e^{s\varphi} (\p_t+v\cdot \nabla_x +\sigma)(e^{-s\varphi}w) = Pw - s B(v) w. 
\end{align*} 
We denote $Q:=(0,T)\times\Omega$.
From the identity
$$
    \int_Q|Pf|^2e^{2s\varphi(t,x)}\,dxdt= \int_Q|L w|^2 \,dxdt,
    $$
applying the integration by parts, one can derive the Carleman estimate in the following proposition. 
\begin{proposition}\label{prop:new Carleman estimate}
    For a fixed $\gamma_1>0$, suppose that $(\sigma,\mu)$ satisfy
    \begin{align}\label{EST:sigma new}
		\sup_{x\in\Omega}B^{-1}(v)|\sigma(x,v)|\leq C_\sigma\quad \hbox{ in }\quad \widetilde V:=\{v:\,|\gamma\cdot v -\beta | \leq \gamma_1\},
	\end{align}
    and  
	\begin{align}\label{CON:mu new condition}
       \sup_{x\in\Omega,\,v\in\mathbb{S}^{d-1}} \int_{\mathbb{S}^{d-1}} |B (v')|^{-2}|\mu(x,v',v)|^2 d\omega(v')\leq C_\mu 
	\end{align}
    for some constants $C_\sigma,\,C_\mu > 0$. 
    Let $f \in H^1(0,T;L^2(S\Omega))$ satisfy $v\cdot\nabla_x f \in L^2(S\Omega_T)$ and $f(T,x,v)=0$. Suppose the initial data $f(0,x,\cdot)$ is supported in $V$. 
    Then there exist positive constants $C=C(a,\gamma_0)$ and $s_0=s_0(d,\gamma_1,C_\sigma,C_\mu, \|\sigma\|_{L^\infty})$ so that for all $s\geq s_0>0$, we have
		\begin{align}\label{EST:new Carleman}
			& s\int_{V}\int_\Omega |f(0,x,v)|^2 e^{2s \varphi(0,x)}\,dxdv + s^2\int_Q\int_{ \mathbb{S}^{d-1}} B^2 |f|^2e^{2s\varphi}\,dxdvdt  \notag\\
			\leq\,& C\int_{S\Omega_T}|\p_tf + v\cdot \nabla_x f + \sigma f - K(f)|^2e^{2s \varphi }\,dxdv dt + Cs \int^T_0\int_{ \mathbb{S}^{d-1}}\int_{\p\Omega}|f|^2e^{2s \varphi }(n(x)\cdot v)\,d\tilde\xi(x,v) dt.
		\end{align}
\end{proposition}
\begin{proof}
     Since $f(T,x,v)=0$, for any vector $v\in \mathbb{S}^{d-1}$, applying the integration by parts leads to the following estimate:
	\begin{align*}
		&\int_Q|L w|^2 \,dxdt \\
        =\,&\int_Q |P w|^2\,dxdt+s^2\int_QB^2|w|^2\,dxdt-2s\int_QBw Pw\,dxdt\\
        \geq\,& s^2\int_QB^2w^2\,dxdt-2s\int_QBw(\p_t w+v\cdot \nabla_x w+\sigma w)\,dxdt\\
        \geq\,& s\int_\Omega B |w(x,v,0)|^2 \,dx - s\int^T_0\int_{\p\Omega} B|w|^2 (n(x)\cdot v)d\lambda(x)\,dt+s^2\int_QB^2 |w|^2 \,dxdt - 2s\int_Q\sigma B|w|^2\,dxdt.
	\end{align*}
    Using \eqref{EST:sigma new},
    we can bound the last term by the third term on the right, that is, 
    $$
        2s\int_Q\sigma B|w|^2\,dxdt\leq {1\over 2} s^2\int_QB^2 |w|^2 \,dxdt
    $$ if $s$ is large enough.
    Since $w=e^{s\varphi}f(t,x,v)$, integrating over $\mathbb{S}^{d-1}$ yields the Carleman estimate without the scattering:
    \begin{align}\label{EST:new Carleman no scattering}
			& s\int_{V}\int_\Omega |f(0,x,v)|^2 e^{2s \varphi(x,0)}\,dxdv + s^2\int_Q\int_{\mathbb{S}^{d-1}} B^2  |f|^2e^{2s\varphi}\,dvdxdt    \notag\\
			\leq\,& C\int_{S\Omega_T}|P f|^2e^{2s \varphi }\,dxdv dt + Cs \int^T_0\int_{\mathbb{S}^{d-1}}\int_{\p\Omega}|f|^2e^{2s \varphi }(n(x)\cdot v)\,d\tilde\xi(x,v) dt.
	\end{align}
    by noting that $B\geq a>0$ in $V$ and $f(0,x,\cdot)$ is supported in $V$.

    To derive \eqref{EST:new Carleman}, we observe that 
    $$
        \int_{S\Omega_T}|P f |^2e^{2s \varphi }\,dxdv dt\leq 2\int_{S\Omega_T}|P f - K(f)|^2e^{2s \varphi }\,dxdv dt + 2\int_{S\Omega_T}|K(f)|^2e^{2s \varphi }\,dxdv dt.
    $$
    Due to $B^{-1}\mu\in L^2(\mathbb{S}^{d-1})$, applying H\"older's inequality, we get
	\begin{align*}
		\LV \int_{\mathbb{S}^{d-1}} \mu(x,v',v)f(x,v',t)\,d v' \RV^2
		\leq \LC\int_{\mathbb{S}^{d-1}} |B(v')|^{-2}|\mu(x,v',v)|^2\,d v' \RC\LC  \int_{\mathbb{S}^{d-1}} |B(v')|^2|f(t,x,v')|^2\,d v' \RC.
	\end{align*}
    It leads to
    \begin{align}\label{EST:carleman k}
        \int_Q\int_{\mathbb{S}^{d-1}} |K(f)|^2 e^{2s\varphi}\,dvdxdt  
    	=\,& \int_Q\int_{\mathbb{S}^{d-1}} \LV \int_{\mathbb{S}^{d-1}} \mu(x,v',v)f(t,x,v')dv'\RV^2 e^{2s\varphi}\,dv dxdt\notag\\
        \leq \,& |\mathbb{S}^{d-1}| C_\mu\int_Q  \LC  \int_{\mathbb{S}^{d-1}} |B(v')|^2|f(t,x,v')|^2\,d v' \RC e^{2s\varphi}\, dxdt,
    \end{align}
    which can then be absorbed by the second term on the left-hand side of \eqref{EST:new Carleman no scattering} provided that $s$ is large enough. This ends the proof. 
\end{proof}

We still need the following energy estimate. It can be showed by adapting the argument in [Lemma~2.1, \cite{Machida14}] for our case $V=\mathbb{S}^{d-1}$ and, therefore, we omit the proof here. 
\begin{lemma}\label{lemma:energy estimate}
   Let $\Omega$ be an open bounded and convex domain with smooth boundary. 
   Suppose that $\sigma\in L^\infty(S\Omega)$ and $\mu\in L^\infty(S\Omega^2)$ satisfy \eqref{EST:sigma} and \eqref{EST:k}. Let $f_0\in L^\infty(S\Omega)$ satisfy $(v\cdot\nabla_x)^\beta f_0\in L^\infty(S\Omega)$, $\beta=1,\,2$. 
   Let $f$ be the solution to the problem 
	\begin{align}\label{EQN: RTE equ IP 00}
		\left\{\begin{array}{rcll}
			\p_tf + v\cdot \nabla_x f + \sigma f &=& K(f) + S & \hbox{in } S\Omega_T,  \\
			f &=&f_0 & \hbox{on }\{0\}\times S\Omega,\\
			f &=&0& \hbox{on }\p_-S\Omega_T,
		\end{array}\right.
	\end{align}
	and satisfy $f \in H^2(0,T;L^2(S\Omega))$ and also $(v\cdot\nabla_x) f \in H^1(0,T;L^2(S\Omega))$. 
   Suppose that the source term has the form
	\[
	S(t,x,v)= \widetilde{S}(x,v)S_0(t,x,v)
	\]
   with $\widetilde S\in L^\infty (S\Omega)$ and 
   $$\|S_0\|_{L^\infty(S\Omega_T)}, \|\p_tS_0\|_{L^\infty(S\Omega_T)}\leq c_3$$
   for some constant $c_3>0$.
    Then there exists a constant $C>0$, which depends on $c_3$, $\|\sigma\|_{L^\infty(S\Omega)}$, and $\|\mu\|_{L^\infty(S\Omega^2)}$, so that
	\begin{align}\label{EST: energy estimate}
		\|\p_t f\|_{L^2(S\Omega)}\leq C \LC \|\widetilde S\|_{L^2(S\Omega)} + \|f_0\|_{L^2(S\Omega)} + \|v\cdot \nabla_x f_0\|_{L^2(S\Omega)}\RC. 
	\end{align} 
\end{lemma}
 
The following theorem states the main estimate which will be used to prove the inverse coefficient/source problems. It indicates that partial information of the source term can be revealed by applying the Carleman estimate on the cut-off function of $\p_t u$ on the time variable, see \cite{Machida14} for a similar argument.

\begin{theorem}\label{thm:recover source Carleman}
Under the same conditions and hypotheses of Lemma~\ref{lemma:energy estimate}, let $S_0(0,x,v)$, $f_0(x,v)$ and $\mu(x,v',\cdot)$ be supported in $V$. Suppose that $\sigma$ and $\mu$ satisfy \eqref{EST:sigma new} and \eqref{CON:mu new condition}. 
Suppose that 
	$$
	0 < c_1\leq S_0(0,x,v)\leq c_2\hbox{ in }\Omega\times V 
	$$ 
    for some fixed constants $c_1,c_2>0$, and
    \begin{align}\label{Assumption:Source}
         \widetilde{S}(x,v)=\widetilde{S}(x,-v)\quad\hbox{ in }S\Omega,\quad \widetilde{S}(x,\cdot)=0\quad 
    \hbox{in }\{v\in\mathbb{S}^{d-1}:\, |\gamma\cdot v|\leq \gamma_0\} .
    \end{align}
Then there exists a positive constant $C$, depending on $c_j$ ($j=1,2,3$), $\|\sigma\|_{L^\infty(S\Omega)}$, and $\|\mu\|_{L^\infty(S\Omega^2)}$, so that
	\begin{align}\label{EST: Carleman diff source introduction}
		\|\widetilde{S} \|_{L^2(S\Omega)}\leq C \LC \|\p_t f\|_{L^2(\p_+S\Omega_T)} + \|f_0\|_{L^2(S\Omega)} + \|v\cdot \nabla_x f_0\|_{L^2(S\Omega)}\RC.
	\end{align} 
\end{theorem}     
\begin{remark}
    From the proof below, one can see that the condition \eqref{Assumption:Source} can be replaced by a slightly relaxed assumption as follows: 
	\begin{align}\label{Assumption:Source general}
	    \int_{\mathbb{S}^{d-1}}|\widetilde{S}(x,v)|^2 \,dv 
	    \leq c_0\int_{V}|\widetilde{S}(x,v)|^2\,dv\quad\hbox{ for all }x\in\Omega 
	\end{align}
    for some fixed constant $c_0>0$. 
     Moreover, the conditions \eqref{EST:sigma new} (with small $\gamma_1$) and \eqref{CON:mu new condition} are satisfied if $\sigma$ and $\mu(\cdot,\cdot,v)$ satisfy \eqref{Assumption:Source}. 
\end{remark}
\begin{proof}
    Let $T$ be large enough so that 
    $$
    T>{\max_{\overline{\Omega}}(\gamma\cdot x) - \min_{\overline{\Omega}}(\gamma\cdot x) \over \beta},
    $$ 
    that is,
    $$
         \max_{\overline{\Omega}} (\gamma\cdot x) < \beta T+ \min_{\overline{\Omega}} (\gamma\cdot x) .
    $$
    This implies
    $$
        \varphi(T, x)=\gamma\cdot x-\beta T\leq \max_{\overline{\Omega}} (\gamma\cdot x) - \beta T<\min_{\overline{\Omega}} (\gamma\cdot x)\leq \varphi(0,x).
    $$
    Due to the continuity of $\varphi$, there exist constants $\zeta>0$, $r_0$ and $r_1$ so that
    $$
        \max_{\overline{\Omega}} (\gamma\cdot x) -\beta T<r_0<r_1<\min_{\overline{\Omega}} (\gamma\cdot x) 
    $$
    and
    $$
    \left\{
    \begin{array}{ll}
        \varphi(t,x) \geq r_1\quad &\hbox{ for }(t,x)\in [0,\zeta]\times \overline\Omega;\\
        \varphi(t,x) \leq r_0\quad &\hbox{ for }(t,x)\in [T-2\zeta,T]\times \overline\Omega.\\
    \end{array}
    \right.
    $$
    
    We consider the function
    $$
        z(t,x,v) = \chi(t)\p_t f(t,x,v),
    $$
    where $f$ is the solution to \eqref{EQN: RTE equ IP 00} and $\chi\in C^\infty_c(\R)$ is a smooth cut-off function satisfying $0\leq \chi \leq1$ and  
    $$
    \chi(t)=\left\{
    \begin{array}{ll}
        1\quad &\hbox{ for }t\in [0,T-2\zeta];\\
        0  \quad &\hbox{ for }t\in [T- \zeta,T] .
    \end{array}
    \right.
    $$
    Hence $z$ satisfies
    $
    z(T,x,v)=0,
    $
    $z|_{\p_-S\Omega_T}=0$, and the nonhomogeneous transport equation
	$$
	    P z-K(z) = \chi \widetilde{S}(\p_t S_0)+(\p_t \chi)\p_t f\quad \hbox{ in }S\Omega_T.
	$$
    Note that since $S_0(0,x,v)$, $f_0(x,v)$ and $\mu(x,v',\cdot)$ are supported in $V$, from \eqref{EQN: RTE equ IP 00}, it follows that the initial data 
    \begin{align}\label{z initial 0}
        z(0,x,v) = \widetilde{S}(x,v)S_0(0,x,v) - v\cdot\nabla_x f_0-\sigma f_0+K(f_0)
    \end{align}
    is also supported in $V$, which satisfies the hypothesis of Proposition~\ref{prop:new Carleman estimate}.
    Now, applying Proposition~\ref{prop:new Carleman estimate} yields that
	\begin{align}\label{EST:z initial}
		s\int_{V}\int_\Omega |z(0,x,v)|^2 e^{2s \varphi(0,x)}\,dxdv  
		\leq C\int_{S\Omega_T}|\chi \widetilde{S}(\p_t S_0)+(\p_t \chi)\p_t f|^2e^{2s \varphi(t,x )}\,dvdx dt + C s \mathcal{D}. 
	\end{align}
    with 
    $$
        \mathcal{D}:=s \int_{\p_+S\Omega_T} |z|^2e^{2s \varphi(t,x)}(n(x)\cdot v)\,d\tilde\xi(x,v)dt\leq e^{C_1s}\|\p_t f\|^2_{L^2(\p_+S\Omega_T)} 
    $$
    for some constant $C_1>0$.
	Next we analyze the first term on the RHS of \eqref{EST:z initial}. To this end, since $\p_t S_0$ is bounded and $\varphi(t,x)\leq \varphi(0,x)$, we obtain 
	\begin{align}\label{EST:z}
		\int_{S\Omega_T}|\chi \widetilde{S}(\p_t S_0)|^2e^{2s \varphi(t,x)}\,dvdx dt
		\leq \,& C\int_{S\Omega_T}|\widetilde{S}|^2e^{2s \varphi(t,x)}\,dvdx dt \notag\\
	    \leq \,& C \int_{S\Omega_T}|\widetilde{S}|^2e^{2s \varphi(0,x)}\,dvdx dt \notag\\
	    \leq \,& C \int_Q\int_{V}|\widetilde{S}|^2e^{2s \varphi(0,x)}\,dvdx dt
	\end{align}
    by applying the assumption \eqref{Assumption:Source}, where $C>0$ depends on  $c_3$.
	In addition, the second term on the RHS of \eqref{EST:z initial} is controlled by applying \eqref{EST: energy estimate} and thus we obtain
	\begin{align}\label{EST:z }
		\int_{S\Omega_T}|(\p_t \chi)\p_t f|^2e^{2s \varphi(t,x)}\,dvdx dt  
        \leq \,& Ce^{2sr_0}\int^{T-\zeta}_{T-2\zeta} \int_{S\Omega} |\p_t f|^2\,dvdxdt \notag\\
        \leq \,& Ce^{2sr_0}\LC\|\widetilde{S}\|^2_{L^2(S\Omega)} + \|f_0\|^2_{L^2(S\Omega)}+\|v\cdot\nabla_x f_0\|^2_{L^2(S\Omega)} \RC 
	\end{align}
    by noting that $\p_t f|_{\p_-S\Omega_T}=0$, $\p_t\chi=0$ in $[0,T-2\zeta]\cup[T-\zeta,T]$ and $\varphi\leq r_0$ in $[T-2\zeta,T]$.
    Furthermore, \eqref{z initial 0} yields
    \begin{align}\label{EST:z 2}
    	&\int_\Omega\int_{V} |z(0,x,v)|^2 e^{2s \varphi(0,x)}\,dxdv +  Ce^{C_1s}\LC\|f_0\|^2_{L^2(S\Omega)}+\|v\cdot\nabla_x f_0\|^2_{L^2(S\Omega)}\RC \notag\\
    	\geq \,& C\int_\Omega\int_{V} |\widetilde{S}(x,v)S_0(0,x,v)|^2e^{2s\varphi(0,x)}\,dvdx.
    \end{align}
	Combining \eqref{EST:z initial}-\eqref{EST:z 2} together, it follows that
	\begin{align*}
		&s\int_\Omega\int_{V} |\widetilde{S}(x,v)S_0(0,x,v)|^2e^{2s\varphi(0,x)}\,dvdx\\
		\leq \,& C\int_Q\int_{V}|\widetilde{S}|^2e^{2s \varphi(0,x)}\,dvdx dt+ Ce^{2sr_0} \|\widetilde{S}\|_{L^2(S\Omega)}^2 + C e^{C_1s}(\|f_0\|^2_{L^2(S\Omega)}+\|v\cdot\nabla_x f_0\|^2_{L^2(S\Omega)}+\mathcal{D}).
	\end{align*}
	Finally, using the facts that \eqref{Assumption:Source}, $S_0(0,x,v)\geq c_1$ in $V$ and $\varphi(0,x)\geq r_1>r_0$, the first two terms on the RHS will be absorbed by the LHS once $s$ is sufficiently large. This results in
	\begin{align*}
		s\int_\Omega\int_{V} |\widetilde{S}(x,v)|^2e^{2sr_1}\,dvdx
		\leq 
		 e^{C_1s}(\|f_0\|^2_{L^2(S\Omega)}+\|v\cdot\nabla_x f_0\|^2_{L^2(S\Omega)}+\mathcal{D}),
	\end{align*}
    which ends the proof.
\end{proof}

\begin{remark}
     In the case that $f_0\equiv 0$, the term $K(f_0)$ in $z(0,x,v)$ vanishes automatically. Hence the assumption on the support of $\mu$ in Theorem~\ref{thm:recover source Carleman} can be removed.
\end{remark}
With Theorem~\ref{thm:recover source Carleman}, we state and prove the uniqueness and stability estimate for the linear coefficients. 

\begin{proposition}\label{thm:recover sigma} 
Let $\Omega$ be an open bounded and convex domain with smooth boundary. 
Suppose that $\sigma_j\in L^\infty(S\Omega)$ and $\mu\in L^\infty(S\Omega^2)$ satisfy \eqref{EST:sigma} and \eqref{EST:k} for $j=1,\,2$. 
Let $N_j:S\Omega\times \R\rightarrow\R$ satisfy the assumption \eqref{condition N} with $q^{(k)}$ replaced by $q^{(k)}_j$ for $j=1,2$, respectively.
For $\varepsilon>0$, let $f_j$ be the unique small solution to
\begin{align}\label{EQN: RTE equ in thm}
	\left\{\begin{array}{rcll}
		\p_tf_j + v\cdot \nabla_x f_j + \sigma_j f_j + N_j(x,v,f_j) &=& K(f_j) & \hbox{in } S\Omega_T,  \\
		f_j &=&\varepsilon h & \hbox{on } \{0\}\times S\Omega,\\
		f_j &=&0& \hbox{on }\p_-S\Omega_T,
	\end{array}\right.
\end{align}
and $F^{(1)}_j=\p_\varepsilon|_{\varepsilon=0} f_j$, $j=1,2$.
If $\sigma_1$, $\sigma_2$ and $\mu(\cdot,\cdot,v)$ satisfy \eqref{Assumption:Source}, 
then
$$
	\|\sigma_1-\sigma_2\|_{L^2(S\Omega)}\leq C \|\p_tF^{(1)}_1-\p_tF^{(1)}_2\|_{L^2(\p_+S\Omega_T)} 
$$
for $h\in L^\infty(S\Omega)$ with support in $V$ satisfying $0<c_1\leq h\leq c_2$ in $\Omega\times V$ for some positive constants $c_1,\,c_2$ and $(v\cdot \nabla_x)^\beta h\in L^\infty(S\Omega)$, $\beta=1,\,2$.

	In particular, if $\mathcal{A}_{\sigma_1,\mu,N_1}(f_0,0)=\mathcal{A}_{\sigma_2,\mu,N_2}(f_0,0)$ for any $(f_0,0)\in \mathcal{X}^\Omega_\delta$, then
$$
\sigma_1=\sigma_2\quad \hbox{in }S\Omega.
$$
\end{proposition}
\begin{proof}
	Let $w^{(1)}:=F_1^{(1)}- F_2^{(1)}$, where $F^{(1)}_j$ is the solution to \eqref{EQN: linear RTE equ IP} with $\sigma$ replaced by $\sigma_j$.  Then $w$ is the solution to   
	\begin{align*}
		\left\{\begin{array}{rcll}
			\p_t w^{(1)} + v\cdot\nabla_x w^{(1)} + \sigma_1 w^{(1)}  &=& K(w^{(1)})-(\sigma_1-\sigma_2)F_2^{(1)}& \hbox{in } S\Omega_T,  \\
			w^{(1)}&=&0 & \hbox{on }\{0\}\times S\Omega,\\
			w^{(1)}&=& 0 & \hbox{on }\p_-S\Omega_T.
		\end{array}\right.
	\end{align*}
	From the hypothesis, we have that $F^{(1)}_2(0,x,v)=h$ is strictly positive in $\Omega\times V$ and also bounded from above in $S\Omega$. Moreover, $F_2^{(1)}$ and $\p_t F_2^{(1)}$ are in $L^\infty(S\Omega_T)$. Indeed, one can see this by taking derivative with respect to $t$ on \eqref{EQN: linear RTE equ IP}:
	 \begin{align}\label{EQN: RTE equ IP sigma}
	 	\left\{\begin{array}{rcll}
	 	 \p_t^2 F^{(1)}_2+(v\cdot\nabla_x) \p_tF^{(1)}_2+\sigma_2\p_tF^{(1)}_2  &=&K(\p_tF^{(1)}_2) & \hbox{in } S\Omega_T,  \\
	 		\p_tF^{(1)}_2 &=&  -v\cdot \nabla_x h-\sigma_2 h + K_2h =:\tilde{h} \in L^\infty(S\Omega_T)& \hbox{on }\{0\}\times S\Omega,\\
	 		\p_tF^{(1)}_2 &=&0 & \hbox{on }\p_-S\Omega_T.
	 	\end{array}\right.
	 \end{align}
	 By the well-posedness result in Proposition~\ref{prop:forward problem linear}, the solution $\p_t F^{(1)}_2$ for \eqref{EQN: RTE equ IP sigma} exists in $L^\infty(S\Omega_T)$ if $(\tilde{h},0)\in \mathcal{X}^\Omega_\delta$.
    Hence, combining these together, there exist positive constants $c_1,\,c_2,\,c_3$ so that
	$$
	0<c_1 \leq F^{(1)}_2(0,x,v)  \leq c_2 \hbox{ in } \Omega\times V,\quad \hbox{and}\quad \| F^{(1)}_2 \|_{L^\infty(S\Omega_T)},\, \|\p_tF^{(1)}_2\|_{L^\infty(S\Omega_T)} \leq c_3.
	$$	
	Then $F^{(1)}_2$, acting as the source $S_0$, satisfies all the conditions in Theorem~\ref{thm:recover source Carleman}.

    In addition, since $\p_tF^{(1)}_2\in L^\infty(S\Omega_T)$ and thus is in $L^2(S\Omega_T)$, from the equation \eqref{EQN: linear RTE equ IP}, we can derive that $v\cdot\nabla_x F^{(1)}_2\in L^2(S\Omega_T)$. Similarly, we can differentiate \eqref{EQN: RTE equ IP sigma} again with respect to $t$ to derive that $\p_t^2 F^{(1)}_2\in L^2(S\Omega_T)$ which leads to $(v\cdot\nabla_x) \p_tF^{(1)}_2\in L^2(S\Omega_T)$. Applying the same argument, one can also deduce that $\p_t^\beta F_1^{(1)}\in  L^2(S\Omega_T)$ with $\beta=1,\,2$, then it implies $(v\cdot\nabla_x)F^{(1)}_1, \,(v\cdot\nabla_x)\p_t F^{(1)}_1\in L^2(S\Omega_T)$. Hence we obtain that $w^{(1)}=F_1^{(1)}- F_2^{(1)}$ satisfies the hypothesis
    $$
    w^{(1)} \in H^2(0,T;L^2(S\Omega)),\quad (v\cdot \nabla_x)w^{(1)}\in H^1(0,T;L^2(S\Omega))
    $$
    in Theorem~\ref{thm:recover source Carleman}.
    We finally get $\|\sigma_1-\sigma_2\|_{L^2(S\Omega)}\leq C \|\p_tF^{(1)}_1-\p_tF^{(1)}_2\|_{L^2(\p_+S\Omega_T)}$ due to Theorem~\ref{thm:recover source Carleman}. Since $\mathcal{A}_{\sigma_1,\mu,N_1}=\mathcal{A}_{\sigma_2,\mu,N_2}$ implies $\p_tF^{(1)}_1=\p_tF^{(1)}_2$ on $\p_+S\Omega_T$, the uniqueness $\sigma_1=\sigma_2$ then holds. 
\end{proof}
\begin{remark}
    In the proposition, we impose the assumption that $(v\cdot \nabla_x)^\beta h\in L^\infty(S\Omega)$ for $\beta=1,2$, so that the term $w^{(1)}=F^{(1)}_1-F^{(1)}_2$ has enough regularity for applying Theorem \ref{thm:recover source Carleman}.
\end{remark}

On the other hand, when $\sigma$ is given, we study below the reconstruction of $\mu$.
\begin{proposition}\label{thm:recover mu} 
	Under the same assumptions as in Proposition~\ref{thm:recover sigma}, suppose that $\sigma\in L^\infty(S\Omega)$ and $\mu_j\in L^\infty(S\Omega^2)$ satisfy \eqref{EST:sigma} and \eqref{EST:k} for $j=1,\,2$. 
    Assume that $\mu_j:=\tilde{\mu}_j(x,v)p(x,v',v)$ with $\tilde\mu_j\in L^\infty(S\Omega)$ and $p(x,v',v)\in L^\infty(S\Omega^2)$ and $p(x,v',v)\geq c>0$ for some positive constant $c$. 
	Let $f_j$ be the unique small solution to
	\begin{align}\label{EQN: RTE equ in thm}
		\left\{\begin{array}{rcll}
			\p_tf_j + v\cdot \nabla_x f_j + \sigma f_j + N_j(x,v,f_j) &=& K_j(f_j) & \hbox{in } S\Omega_T,  \\
			f_j &=&\varepsilon h & \hbox{on } \{0\}\times S\Omega,\\
			f_j &=& 0& \hbox{on }\p_-S\Omega_T,
		\end{array}\right.
	\end{align}
	and $F^{(1)}_j=\p_\varepsilon|_{\varepsilon=0} f_j$, $j=1,2$. If $\sigma$, $\tilde \mu_1$, $\tilde\mu_2$ and $p(\cdot,\cdot,v)$ satisfy \eqref{Assumption:Source}, then 
	$$
	\|\tilde{\mu}_1-\tilde{\mu}_2\|_{L^2(S\Omega)}\leq C \|\p_tF^{(1)}_1-\p_tF^{(1)}_2\|_{L^2(\p_+S\Omega_T)} 
	$$
	for $h\in L^\infty(S\Omega)$ with support in $V$ satisfying $0<c_1\leq h\leq c_2$ in $\Omega\times V$
    for some positive constants $c_1,\,c_2$ and $(v\cdot \nabla_x)^\beta h\in L^\infty(S\Omega)$, $\beta=1,\,2$.
	
	In particular, if $\mathcal{A}_{\sigma,\mu_1,N_1}(f_0,0)=\mathcal{A}_{\sigma,\mu_2,N_2}(f_0,0)$ for any $(f_0,0)\in \mathcal{X}^\Omega_\delta$, then
	$$
	\tilde{\mu}_1=\tilde{\mu}_2 \quad \hbox{in }S\Omega.
	$$
\end{proposition}
\begin{proof}
	Let $w^{(1)}:=F_1^{(1)}- F_2^{(1)}$, where $F^{(1)}_j$ is the solution to \eqref{EQN: linear RTE equ IP} with $(\sigma,\mu)$ replaced by $(\sigma,\mu_j)$.  Then $w^{(1)}$ is the solution to   
	\begin{align*}
		\left\{\begin{array}{rcll}
			\p_t w^{(1)} + v\cdot\nabla_x w^{(1)} + \sigma w^{(1)}  &=& K_1(w^{(1)})+(K_1-K_2)F_2^{(1)}& \hbox{in } S\Omega_T,  \\
			w^{(1)} &=&0 & \hbox{on }\{0\}\times S\Omega,\\
			w^{(1)} &=& 0 & \hbox{on }\p_-S\Omega_T.
		\end{array}\right.
	\end{align*}
	The source term is $(K_1-K_2)F_2^{(1)}= (\tilde{\mu}_1-\tilde{\mu}_2)(x,v) \int p(x,v',v) F_2^{(1)}(t,x,v') dv'$.  
	Following a similar argument as in the proof of Proposition~\ref{thm:recover sigma}, we can deduce that 
    $\tilde{\mu}_1=\tilde{\mu}_2$ by applying Theorem~\ref{thm:recover source Carleman}.
\end{proof}

\subsection{Recover the nonlinear term}\label{sec:taylor}
From Section~\ref{sec:sigma mu}, we have discussed how to reconstruct one unknown linear coefficient from the measurement $\mathcal{A}_{\sigma,\mu,N}$ provided that the other one is given. Therefore, in this section, we suppose that $(\sigma,\,\mu)$ are recovered and only focus on the reconstruction of the nonlinear term.

The setting is as follows. Suppose that $f_j$, $j=1,2$ are the solutions to  
\begin{align}\label{EQN: RTE equ IP simultaneous j}
	\left\{\begin{array}{rcll}
		\p_tf_j + v\cdot \nabla_x f_j + \sigma f_j + N_j(x,v,f_j) &=& K(f_j) & \hbox{in } S\Omega_T,  \\
		f_j&=& \varepsilon h & \hbox{on }\{0\}\times S\Omega,\\
		f_j&=& 0& \hbox{on }\p_-S\Omega_T,
	\end{array}\right.
\end{align}
where the nonlinear term $N_j$ satisfy $N_j(x,v,f)= \sum_{k=2}^\infty q_j^{(k)}(x,v){f^k\over k!}$.

To recover $N_j$, it is sufficient to recover every $q_j^{(k)}$, $k\geq 2$. To this end, we apply the induction argument and also rely on the higher order linearization technique to extract out the information of $q_j^{(k)}$ from the measurement. 

To simplify the notation, we denote the operator $\mathcal{T}$ by
$$
\mathcal{T}:= \p_t+v\cdot\nabla_x+\sigma-K.
$$
Recall that $F^{(k)}_j=\p^k_\varepsilon|_{\varepsilon=0} f_j$. When $k=2$, the function $F_j^{(2)}$, $j=1,2$, satisfies the problem
\begin{align}\label{EQN: RTE equ IP linear 1 F2 simultaneous}
	\left\{\begin{array}{rcll}
		\mathcal{T} F^{(2)}_j &=& - q_j^{(2)}(x,v) (F^{(1)})^2 & \hbox{in } S\Omega_T,  \\
		F^{(2)}_j &=& 0 & \hbox{on }\{0\}\times S\Omega,\\
		F^{(2)}_j &=&0 & \hbox{on }\p_-S\Omega_T 
	\end{array}\right.
\end{align}
due to the well-posed result $f_j(t,x,v;0)=0$.
Notice that since both $F^{(1)}_j$, $j=1,2$ satisfy \eqref{EQN: linear RTE equ IP} with the same data, the well-posedness for the initial boundary value problem for the transport equation yields that $F^{(1)}:=F^{(1)}_1=F^{(1)}_2$.

We are ready to recover the coefficient $q^{(2)}_j$. 
\begin{lemma}\label{lemma:q2}
Suppose that $\sigma\in L^\infty(S\Omega)$ and $\mu\in L^\infty(S\Omega^2)$ satisfy \eqref{EST:sigma} and \eqref{EST:k}. Let $\sigma$, $\mu(\cdot,\cdot,v)$, $q_1^{(2)}$ and $q_2^{(2)}$ satisfy \eqref{Assumption:Source}.
If $h\in L^\infty(S\Omega)$ with support in $V$ satisfying $0<c_1\leq h\leq c_2$ in $\Omega\times V$ for some positive constants $c_1,\,c_2$ and $(v\cdot \nabla_x)^\beta h\in L^\infty(S\Omega)$, $\beta=1,\,2$, then
	\begin{align}\label{EST: diff source introduction simultaneous 2}
		\|q_1^{(2)}-q_2^{(2)}\|_{L^2(S\Omega)}\leq C \|\p_t F^{(2)}_1-\p_t F^{(2)}_2\|_{L^2(\p_+S\Omega_T)}  
	\end{align} 
for some constant $C>0$.
\end{lemma}
\begin{proof}
	 Let $w^{(2)}: = F^{(2)}_1-F^{(2)}_2$ and then $w^{(2)} \in L^\infty(S\Omega_T)$ satisfies
	 \begin{align*} 
	 	\left\{\begin{array}{rcll}
	 		\mathcal{T} w^{(2)}  &=& - (q_1^{(2)}-q^{(2)}_2)(x,v) (F^{(1)})^2 & \hbox{in } S\Omega_T,  \\
	 		w^{(2)}  &=& 0 & \hbox{on }\{0\}\times S\Omega,\\
	 		w^{(2)}  &=&0 & \hbox{on }\p_-S\Omega_T.
	 	\end{array}\right.
	 \end{align*}
    Following a similiar argument as in Proposition~\ref{thm:recover sigma}, Proposition~\ref{prop:forward problem linear} yields that $F^{(1)}$ and $\p_t F^{(1)}$ are in $L^\infty(S\Omega_T)$.
    Therefore, we can derive that there exist positive constants $c_1,\,c_2,\,c_3$ so that
	$$
	0<c_1 \leq F^{(1)}(0,x,v)\leq c_2 \hbox{ in } \Omega\times V,\quad \hbox{and}\quad \|(F^{(1)})^2\|_{L^\infty(S\Omega_T)},\, \|\p_t (F^{(1)})^2\|_{L^\infty(S\Omega_T)} \leq c_3.
	$$	
	Then Theorem~\ref{thm:recover source Carleman} leads to the estimate \eqref{EST: diff source introduction simultaneous 2} immediately.
\end{proof}

Since $\mathcal{A}_{\sigma,\mu,N_1}=\mathcal{A}_{\sigma,\mu,N_2}$ implies $\p_t F^{(2)}_1=\p_t F^{(2)}_2$ on $\p_+S\Omega_T$, from Lemma~\ref{lemma:q2}, it suggests that 
$$
    q^{(2)}:=q^{(2)}_1=q^{(2)}_2
$$ 
when the boundary measurements are the same.

To recover the higher order terms $q^{(m)}_j$, $m > 2$, notice that the function $F^{(m)}_j$, $j=1,2$ satisfy the problem
\begin{align}\label{EQN: RTE equ IP linear Fm simultaneous 0}
	\left\{\begin{array}{rcll}
		\mathcal{T} F^{(m)}_j &=&  -q^{(m)}_j (x,v)(F^{(1)})^m + R_{m,j} & \hbox{in } S\Omega_T,  \\
		F^{(m)}_j&=& 0 & \hbox{on }\{0\}\times S\Omega,\\
		F^{(m)}_j&=&0 & \hbox{on }\p_-S\Omega_T,
	\end{array}\right.
\end{align}
where 
\begin{align}\label{DEF:Q}
    R_{m,j}(t,x,v):= \p_\varepsilon^m \LC \sum_{k=2}^{m-1} q_j^{(k)}{f^k_j\over k!}\RC \Big|_{\varepsilon=0} .
\end{align}
Notice that the remainder term $R_{m,j}$ only contains the derivatives of $f^k_j$ up to order $m-1$, that is, $F^{(1)}_j,\ldots, F^{(m-1)}_j$, and also $q^{(2)}_j,\ldots,q^{(m-1)}_j$.

\begin{lemma}\label{lemma:unique F}
Let $m\geq 3$. Suppose that $\mathcal{A}_{\sigma,\mu,N_1}(f_0,0)=\mathcal{A}_{\sigma,\mu,N_2}(f_0,0)$ for all $(f_0,0)\in \mathcal{X}^\Omega_\delta$ and also $q^{(k)}_1= q^{(k)}_2$ for $k=2,\ldots, m-1$. Then for any $1\leq k \leq m-1$, we have
\begin{align*}
	F^{(k)}_1 = F^{(k)}_2 \quad\hbox{ in }S\Omega_T.
\end{align*}
\end{lemma}
\begin{proof}
	We proceed by applying the induction argument. First we consider the case $m=3$. Since $\mathcal{A}_{\sigma,\mu,N_1}=\mathcal{A}_{\sigma,\mu,N_2}$, we have $q^{(2)}:=q^{(2)}_1= q^{(2)}_2$ due to Lemma~\ref{lemma:q2}. Based on this,
	$F^{(2)}_1$ and $F^{(2)}_2$ satisfy the same initial boundary value problem with the same source $q^{(2)}(x,v) (F^{(1)})^2$. The well-posedness theorem yields that 
	$$
	F^{(2)}:=F^{(2)}_1 = F^{(2)}_2\quad \hbox{ in }S\Omega_T.
	$$
	Hence the case $m=3$ holds.
	
	Next by the induction, suppose that if $q^{(k)}_1= q^{(k)}_2$ for $k=2,\ldots, m-2$, then $F^{(k)}_1 = F^{(k)}_2$ in $S\Omega_T$, $1\leq k\leq m-2$, holds. 
	It is sufficient to show that $F^{(m-1)}_1 = F^{(m-1)}_2$ when $q^{(k)}_1= q^{(k)}_2$ for $k=2,\ldots, m-1$.
	To this end, we observe that $F^{(m-1)}_j$ satisfy 
	\begin{align}\label{EQN: RTE equ IP linear Fm-1 simultaneous}
		\left\{\begin{array}{rcll}
			\mathcal{T} F^{(m-1)}_j &=&  -q^{(m-1)}_j (x,v)(F^{(1)})^{m-1} + R_{m-1,j} & \hbox{in } S\Omega_T,  \\
			F^{(m-1)}_j&=& 0 & \hbox{on }\{0\}\times S\Omega,\\
			F^{(m-1)}_j&=&0 & \hbox{on }\p_-S\Omega_T.
		\end{array}\right.
	\end{align}
    It is clear that $q^{(m-1)}_1 (x,v)(F^{(1)})^{m-1}= q^{(m-1)}_2 (x,v)(F^{(1)})^{m-1}$ and $R_{m-1,1}=R_{m-1,2}$ by applying $q^{(k)}_1= q^{(k)}_2$ for $k=2,\ldots, m-1$ and the definition of $R_{m-1,j}$, which only contains $F^{(1)}_j,\ldots, F^{(m-2)}_j$ and $q^{(2)}_j,\ldots,q^{(m-2)}_j$.
	Therefore, $F^{(m-1)}_j$ satisfies the same initial boundary value problem with the same source, which then leads to
	$F^{(m-1)}_1 = F^{(m-1)}_2$ due to the well-posedness theorem again. This completes the proof.
\end{proof}

With Lemma~\ref{lemma:q2} and Lemma~\ref{lemma:unique F}, we can now stably and uniquely recover all the terms $q^{(k)}_j$ for all $k\geq 2$. 
\begin{lemma}\label{lemma:q m}
	Suppose all conditions in Lemma~\ref{lemma:q2} and Lemma~\ref{lemma:unique F} hold. Let $q_1^{(m)}$ and $q_2^{(m)}$ satisfy \eqref{Assumption:Source} for $m\geq 2$. 
If $\mathcal{A}_{\sigma,\mu,N_1}(f_0,0)=\mathcal{A}_{\sigma,\mu,N_2}(f_0,0)$ for all $(f_0,0)\in \mathcal{X}^\Omega_\delta$, then 
	$$
	q_1^{(m)}=q_2^{(m)} \quad \hbox{ for all }m\geq 2.
	$$
\end{lemma}
\begin{proof}
For any fixed positive integer $m\geq 2$, we will show that $q_1^{(m)}=q_2^{(m)}$ by applying the induction argument. Recall that we have shown the case $m=2$, that is, $q_1^{(2)}=q_2^{(2)}$. By the induction argument, we suppose that $q_1^{(k)}=q_2^{(k)}$ for all $2\leq k\leq m-1$. The objective is to prove  $q_1^{(m)}=q_2^{(m)}$.
From Lemma~\ref{lemma:unique F}, we can derive that $F^{(k)}_1 = F^{(k)}_2$ in $S\Omega_T$ for $k=1,\ldots, m-1$. This implies that $R_{m,1} = R_{m,2}$ by the definition of $R_{m,j}$. Hence, we derive that
\begin{align}\label{EQN: RTE equ IP linear Fm simultaneous}
	\left\{\begin{array}{rcll}
		\mathcal{T} (F^{(m)}_1-F^{(m)}_2) &=&  -\LC q^{(m)}_1-q^{(m)}_2 \RC (x,v)(F^{(1)})^m & \hbox{in } S\Omega_T,  \\
		F^{(m)}_1-F^{(m)}_2 &=& 0 & \hbox{on }\{0\}\times S\Omega,\\
		F^{(m)}_1-F^{(m)}_2 &=&0 & \hbox{on }\p_-S\Omega_T.
	\end{array}\right.
\end{align}
Note that as discussed above, for sufficiently small and well-chosen data $h>0$, there exist positive constants $c_1,\, c_2,\,c_3$ so that
$$
0<c_1 \leq F^{(1)}(0,x,v)\leq c_2 \hbox{ in }\Omega\times V,\quad \hbox{and}\quad \|(F^{(1)})^m\|_{L^\infty(S\Omega_T)},\, \|\p_t (F^{(1)})^m\|_{L^\infty(S\Omega_T)} \leq c_3 
$$	
for any integer $m\geq 2$. With these estimates, we can apply the Carleman estimate again in Theorem~\ref{thm:recover source Carleman} to the problem \eqref{EQN: RTE equ IP linear Fm simultaneous} to recover the $m$-th order term, namely,
\begin{align}\label{EST: diff source introduction simultaneous 1}
	\|q_1^{(m)}-q_2^{(m)}\|_{L^2(S\Omega)}\leq C \|\p_t F^{(m)}_1-\p_t F^{(m)}_2\|_{L^2(\p_+S\Omega_T)}  
\end{align} 
for some constant $C>0$. Thus $q_1^{(m)}=q_2^{(m)}$ follows by the fact that $\p_t F^{(m)}_1=\p_t F^{(m)}_2$ on $\p_+S\Omega_T$.
\end{proof}

Finally, we prove Theorem~\ref{thm:into Euclidean}
\begin{proof}[Proof of Theorem~\ref{thm:into Euclidean}]
    Since $\mathcal{A}_{\sigma_1,\mu,N_1}=\mathcal{A}_{\sigma_2,\mu,N_2}$ implies $\p_t F^{(m)}_1 = \p_t F^{(m)}_2$ on $\p_+S\Omega_T$ for $m\geq 2$, with \eqref{condition N Taylor}, Proposition~\ref{thm:recover sigma} and Lemma~\ref{lemma:q m} immediately yield the result.
\end{proof}


\section{Inverse problems on Riemannian manifolds}\label{sec:Riemannian}
Let $M$ be the interior of a compact Riemannian manifold $\overline M$ with strictly convex boundary $\p M$, of dimension $\geq 2$. 
Let $f$ be the solution to the problem 
\begin{align}\label{EQN: NRTE equ IP Riemannian}
	\left\{\begin{array}{rcll}
		\p_tf + X f + \sigma f + N(x,v,f) &=& 0& \hbox{in } SM_T,  \\
		f&=& f_0& \hbox{on }\{0\}\times SM,\\
		f&=& f_- & \hbox{on } \p_-SM_T.
	\end{array}\right.
\end{align}

The objective of the section is to recover $\sigma$ and $N(x,v,f)$. For this purpose, we will deduce the Carleman estimate and energy estimate on the Riemannian manifolds. Since we could not find the relative results for the transport equation on manifolds in the literature, we prove them in the upcoming subsection. Once these are established, we will then turn to the determination of $\sigma$ and $N$.

\subsection{Carleman estimate on Riemannian Manifolds}
Let $\sigma\in L^\infty(SM)$, we denote the operators
$$Pu:=\p_t u +Xu+\sigma u,\quad P_0:=\p_t+X,$$
where $X$ is the geodesic vector field on $SM$. Recall that $\tau_+(x,v)$ is the forward exit time of the geodesic starting at $(x,v)\in SM$.
Since the manifold is non-trapping, there exists $D>0$, such that $0<\tau_+(x,v)<D$ for all $(x,v)\in SM$. In particular, we let $D$ be the least upper bound for $\tau_+$ on $SM$.

Since
$$
X\tau_+(x,v)=\frac{d}{dk} \tau_+(\phi_k(x,v))|_{k=0},
$$ 
and $\tau_+(\phi_t (x,v)) = \tau_+(x,v)-t$, we have
$$X\tau_+(x,v)=\lim_{t\to 0}\frac{\tau_+(\phi_t(x,v))-\tau_+(x,v)}{t}=\lim_{t\to 0}\frac{-t}{t}=-1.$$
In particular, $\tau_+$ is a smooth function on $SM$.

We define the phase function $\varphi$ by
$$\varphi(t,x,v)=-\beta t-\tau_+(x,v),$$
for some positive constant $\beta$, so that
$$P_0 \varphi=-\beta+1=:B>0\quad \hbox{if $\beta<1$}.$$

We first deduce the Carleman estimate for the transport equation on a Riemannian manifold.  
\begin{theorem}[Carleman estimate]\label{Carleman estimate Riemannian}
	Let $\sigma\in L^\infty(SM)$. There exists $s_0$ and $C>0$, such that for all $s\geq s_0>0$
	\begin{equation*}
		\begin{split}
			&  C\int_0^T\int_{SM} e^{2s\varphi}|Pu|^2\, d\Sigma dt\\
			\geq\,&  Cs^2 \int_0^T\int_{SM} e^{2s\varphi} u^2\, d\Sigma dt+s \int_{SM} e^{2s\varphi(0,x,v)}u^2(0,x,v)\, d\Sigma-s \int_{SM} e^{2s\varphi(T,x,v)}u^2(T,x,v)\, d\Sigma\\
			&  -s\int_0^T\int_{\p SM} e^{2s\varphi}u^2\, d\xi(x,v) dt,
		\end{split}
	\end{equation*}
	for $u\in H^1(0,T;L^2(SM))$ and $Xu\in L^2(SM_T)$.
	Here $d\Sigma=d\Sigma(x,v)$ the volume form of $SM$, $d\xi(x,v)=\left<v,n(x)\right>_{g(x)}d\tilde\xi(x,v)$ with $n(x)$ the unit outer normal vector at $x\in \p M$ and $d\tilde{\xi}$ the volume form of $\p SM$.
\end{theorem}

\begin{proof}
	Now let $w(t,x,v)=e^{s\varphi(t,x,v)} u(t,x,v)$, we define
	$$Lw:=e^{s\varphi}P_0(e^{-s\varphi}w)=P_0 w-sBw.$$
	We integrate $Lw$ over $[0,T]\times SM$ to get
	\begin{align*}
		& \int_0^T\int_{SM} |Lw|^2\, d\Sigma dt=\int_0^T\int_{SM} |P_0w-sBw|^2\, d\Sigma dt\\
		=\,& \int_0^T\int_{SM} |P_0w|^2\, d\Sigma dt+s^2B^2\int_0^T\int_{SM}|w|^2\,d\Sigma dt-2sB\int_0^T\int_{SM}w (P_0w)\, d\Sigma dt\\
		\geq \,& s^2B^2\int_0^T\int_{SM}|w|^2\, d\Sigma dt-2sB\int_0^T\int_{SM}w (P_0w)\, d\Sigma dt\\
		= \,& s^2B^2\int_0^T\int_{SM}|w|^2\, d\Sigma dt-sB\int_0^T\int_{SM} \p_t (w^2)\, d\Sigma dt-sB\int_0^T\int_{SM} X (w^2)\, d\Sigma dt\\
		=\,& s^2B^2\int_0^T\int_{SM}|w|^2\, d\Sigma dt+sB \int_{SM} w^2(0,x,v)\, d\Sigma-sB \int_{SM} w^2(T,x,v)\, d\Sigma\\
		& -sB\int_0^T\int_{\p SM} w^2\, d\xi(x,v)dt.
	\end{align*}

	Notice that $Lw=e^{s\varphi} P_0 u$ and $w =e^{s\varphi(t,x,v)} u$, the above calculation gives
	\begin{align*}
		& \int_0^T\int_{SM} e^{2s\varphi}|P_0u|^2\, d\Sigma dt\\
		\geq \, & s^2B^2\int_0^T\int_{SM} e^{2s\varphi} u^2\, d\Sigma dt+sB \int_{SM} e^{2s\varphi(0,x,v)}u^2(0,x,v)\, d\Sigma-sB \int_{SM} e^{2s\varphi(T,x,v)}u^2(T,x,v)\, d\Sigma\\
		& -sB\int_0^T\int_{\p SM} e^{2s\varphi}u^2\, d\xi(x,v)dt.
	\end{align*}
	
	To incorporate the absorbing coefficient $\sigma$, observe the following
	$$|P_0u|^2=|Pu-\sigma u|^2\leq 2|Pu|^2+2|\sigma u|^2,$$
	which yields that 
	\begin{align*}
		& 2\int_0^T\int_{SM} e^{2s\varphi}|Pu|^2\, d\Sigma dt+2\int_0^T\int_{SM} e^{2s\varphi}|\sigma u|^2\, d\Sigma dt\\
		\geq \,& s^2B^2\int_0^T\int_{SM} e^{2s\varphi} u^2\, d\Sigma dt+sB \int_{SM} e^{2s\varphi(0,x,v)}u^2(0,x,v)\, d\Sigma-sB \int_{SM} e^{2s\varphi(T,x,v)}u^2(T,x,v)\, d\Sigma\\
		&  -sB\int_0^T\int_{\p SM} e^{2s\varphi}u^2\, d\xi(x,v)dt.
	\end{align*}
	Since $\sigma\in L^\infty(SM)$, by choosing sufficiently large $s$, the second term on the left-hand side of the above inequality can be absorbed by the first term on the right-hand side, it follows that there exist $s_0$ and $C>0$ (independent of $s$), for all $s\geq s_0$, 
	\begin{align*}
		& C\int_0^T\int_{SM} e^{2s\varphi}|Pu|^2\, d\Sigma dt\\
		\geq \, & Cs^2 \int_0^T\int_{SM} e^{2s\varphi} u^2\, d\Sigma dt+s \int_{SM} e^{2s\varphi(0,x,v)}u^2(0,x,v)\, d\Sigma-s \int_{SM} e^{2s\varphi(T,x,v)}u^2(T,x,v)\, d\Sigma\\
		&  -s\int_0^T\int_{\p SM} e^{2s\varphi}u^2\, d\xi(x,v)dt.
	\end{align*}
\end{proof}

\begin{remark}\label{remark no scattering}
    It is worth mentioning that different from Proposition \ref{prop:new Carleman estimate} for the Euclidean case, the Carleman estimate on Riemannian manifolds does not contain the scattering term. This is due to the fact that our weight function $\varphi=-\beta t-\tau_+$ depends on the direction $v$. In the proof of Proposition \ref{prop:new Carleman estimate}, see also \cite{Machida14}, it is essential that the weight function $\varphi$ is independent of $v$, so that the integral $\int_0^T \int_{SM} e^{2s\varphi}|K(u)|^2\,d\Sigma dt$ can be absorbed by the term $s^2\int_0^T\int_{SM}e^{2s\varphi}u^2\,d\Sigma dt$ for $s>0$ sufficiently large. 
    
    On the other hand, if we replace $-\tau_+$ by some globally defined function $\psi(x)$ independent of $v$, then $X\psi(x,v)=\left<v,\nabla \psi(x)\right>_{g(x)}$ can not always be positive. In this case, $(\p_t+X)\varphi=-\beta +X\psi$ could be negative on $SM_T$, consequently the Carleman estimate can not hold for such weight function. Therefore, one can not expect to find a globally defined Carleman weight independent of $v$ to prove similar Carleman estimates. 
\end{remark}

Next, we derive an energy estimate, which will be used to establish uniqueness and stability results for an inverse source problem of linear transport equations on manifolds later.

\begin{lemma}\label{energy estimate Riemannian}
	Suppose $\sigma\in L^\infty(SM)$ satisfies \eqref{EST:sigma}. Let $f_0\in L^\infty(SM)$ satisfy $X f_0\in L^\infty(SM)$, and $f_-\in L^\infty(\p_-SM_T)$ satisfy $\p_t f_-\in L^\infty(\p_-SM_T)$. Suppose that the source term has the form
	\[
	S(t,x,v)= \widetilde{S}(x,v)S_0(t,x,v)
	\]
	with $\widetilde{S}\in L^\infty(SM)$,  
	 $$
	\|S_0(0,\cdot,\cdot)\|_{L^\infty( SM)},\quad  \|\p_tS_0\|_{L^\infty( SM_T)}\leq c,
	$$ 
	for some fixed constant $c>0$. 
	Let $f$ be the solution to the problem 
	\begin{align}\label{EQN: RTE equ IP Riemannian 2}
		\left\{\begin{array}{rcll}
			\p_tf + X f + \sigma f &=& S & \hbox{in } SM_T,  \\
			f &=&f_0 & \hbox{on }\{0\}\times SM,\\
			f &=&  f_- & \hbox{on } \p_-SM_T.
		\end{array}\right.
	\end{align}
	Then there exists a positive constant $C$, depending on $c$, $T$, and $\|\sigma\|_{L^\infty(SM)}$, 
	so that
	\begin{align}\label{EST: energy estimate 1}
		\|\p_t f\|_{L^2(SM)}\leq C\LC \|\tilde S\|_{L^2(SM)}+\|f_0\|_{L^2(SM)}+\|X f_0\|_{L^2(SM)}+ \|\p_t f_-\|_{L^2(\p_-SM_T)}\RC
	\end{align} 
	for any $0\leq t\leq T$ and
	\begin{align}\label{EST: energy estimate 2}
		\|\p_t f\|_{L^2( \p_+SM_T)}\leq C\LC\|\tilde S\|_{L^2(SM)}+\|f_0\|_{L^2(SM)}+\|X f_0\|_{L^2(SM)}+  \|\p_t f_-\|_{L^2(\p_-SM_T)}\RC 
	\end{align} 
	for $f\in H^2(0,T;L^2(SM))$ and $Xf\in H^1(0,T;L^2(SM))$.
\end{lemma}

\begin{proof}
	Taking derivative of the transport equation with respect to the time $t$ gives
	\begin{equation}\label{t derivative of transport equation}
		\p_t(\p_t f)+X(\p_t f)+\sigma (\p_t f)=\tilde S(x,v)\p_t S_0(t,x,v).
	\end{equation}
	Then we multiply $2\p_t f$ to \eqref{t derivative of transport equation} and integrate over $SM$ to get
	\begin{equation}\label{EST: energy estimate 3}
		\begin{split}
			\p_t \int_{SM} |\p_t f|^2\, d\Sigma & =-\int_{SM} X(|\p_t f|^2)\, d\Sigma-2\int_{SM} \sigma|\p_t f|^2\, d\Sigma+2\int_{SM}\tilde S (\p_t S_0)\p_t f\, d\Sigma\\
			& \leq -\int_{\p SM}|\p_t f|^2\, d\xi(x,v)+C\int_{SM} |\p_t f|^2\, d\Sigma+ C\int_{SM} |\tilde S|^2\,d\Sigma\\
			&  \leq -\int_{\p_- SM}|\p_t f|^2\, d\xi(x,v)+ C\int_{SM} |\p_t f|^2\, d\Sigma+ C\int_{SM} |\tilde S|^2\,d\Sigma,
		\end{split}
	\end{equation}
	where the constant $C>0$ depends on $\sigma$ and $c$.
	Here we are using the fact that $\int_{\p_+ SM}|\p_t f|^2\, d\xi(x,v)\geq 0$.
	We denote $E(t)=\int_{SM}|\p_t f|^2(t,x,v)\,d\Sigma$, integrate \eqref{EST: energy estimate 3} over the time interval $(0,t)$, then
	\begin{equation*}
		E(t)-E(0)\leq C\int_0^t E(s)\,ds+ \|\p_t f_-\|^2_{L^2(\p_-SM_T)}+ CT\|\tilde S\|^2_{L^2(SM)}
	\end{equation*}
	for $0\leq t\leq T$. In the meantime, let $t=0$ in the transport equation, we obtain $\p_tf(0,x,v)+Xf_0+\sigma f_0=S(0,x,v)$, which gives
	$$E(0)=\int_{SM}|-Xf_0-\sigma f_0+S (0,x,v) |^2\,d\Sigma\leq C\LC \|Xf_0\|^2_{L^2(SM)}+\|f_0\|^2_{L^2(SM)}+\|\tilde S\|^2_{L^2(SM)}\RC.$$
	Therefore 
	\begin{equation*}
		\begin{split}
			E(t)\leq C\int_0^t E(s)\,ds+ C\LC\|Xf_0\|^2_{L^2(SM)}+\|f_0\|^2_{L^2(SM)}+ \|\p_t f_-\|^2_{L^2(\p_-SM_T)}+\|\tilde S\|^2_{L^2(SM)}\RC,
		\end{split}
	\end{equation*}
	where $C$ depends on $\sigma,\,c$ and $T$.
	We apply the Gronwall's inequality to get
	$$E(t)\leq Ce^T \LC\|Xf_0\|^2_{L^2(SM)}+\|f_0\|^2_{L^2(SM)}+ \|\p_t f_-\|^2_{L^2(\p_-SM_T)}+\|\tilde S\|^2_{L^2(SM)}\RC,$$
	which proves the estimate \eqref{EST: energy estimate 1}.
	
	To prove \eqref{EST: energy estimate 2}, we return to \eqref{EST: energy estimate 3} and integrate it over $(0,T)$, by \eqref{EST: energy estimate 1}, we obtain
	\begin{equation*}
		\begin{split}
			\int_0^T \int_{\p_+SM}|\p_t f|^2\,d\xi(x,v)dt & \leq E(0)-E(T)+C\int_0^T E(s)\,ds+ CT\|\tilde S\|^2_{L^2(SM)}+ \|\p_t f_-\|^2_{L^2(\p_-SM_T) }\\
			& \leq E(0)+C\int_0^T E(s)\,ds+ CT\|\tilde S\|^2_{L^2(SM)}+ \|\p_t f_-\|^2_{L^2(\p_-SM_T)} \\
			& \leq C\LC\|Xf_0\|^2_{L^2(SM)}+\|f_0\|^2_{L^2(SM)}+\|\tilde S\|^2_{L^2(SM)}+ \|\p_t f_-\|^2_{L^2(\p_-SM_T)} \RC,
		\end{split}
	\end{equation*}
	where $C$ depends on $\sigma,\,c$ and $T$. 
\end{proof}

Finally, we will apply the Carleman estimate in Theorem~\ref{Carleman estimate Riemannian} and Lemma~\ref{energy estimate Riemannian} to control the source term in the transport equation.

\begin{theorem}\label{Thm:estimate tilde S Riemann}
	Suppose $\sigma\in L^\infty(SM)$ satisfies \eqref{EST:sigma}. Let $f_0\in L^\infty(SM)$ satisfy $X^\beta f_0\in L^\infty(SM)$ with $\beta=1,\,2$, and $f_-$ satisfy $\p_t f_-\in L^\infty(\p_-SM_T)$. Suppose that the source term has the form of 
	\[
	S(t,x,v)= \widetilde{S}(x,v)S_0(t,x,v)
	\]
	with $\widetilde{S}\in L^\infty(SM)$, 
	$$
	0 < c_1\leq S_0(0,x,v) \leq c_2\quad \hbox{ in } SM  \quad \hbox{and}\quad\|S_0\|_{L^\infty(SM_T)}, \|\p_tS_0\|_{L^\infty(SM_T)}\leq c_3,
	$$ 
	for some fixed constants $c_1,c_2,c_3>0$.
	Let $f$ be the solution to the problem 
	\begin{align}\label{EQN: RTE equ IP Riemannian 1}
		\left\{\begin{array}{rcll}
			\p_tf + X f + \sigma f &=& S & \hbox{in } SM_T,  \\
			f &=&f_0 & \hbox{on }\{0\}\times SM,\\
			f &=&  f_-  & \hbox{on } \p_-SM_T.
		\end{array}\right.
	\end{align}
	Then there exists a positive constant $C$, depending on $c_1$, $c_2$, $c_3$ and $\|\sigma\|_{L^\infty(SM)}$, 
	so that
	\begin{align}\label{EST: Carleman diff source Riemannian}
		\|\widetilde{S} \|_{L^2(SM)}\leq C \LC \|\p_t f\|_{L^2(\p_+SM_T)} + \|f_0\|_{L^2(SM)} + \|X f_0\|_{L^2(SM)}+ \|\p_t f_-\|_{L^2(\p_-SM_T)}\RC 
	\end{align} 
	for $f\in H^2(0,T;L^2(SM))$ and $Xf\in H^1(0,T;L^2(SM))$.
\end{theorem}

\begin{proof}
	We choose $T>D/\beta$, where $D$ is the least upper bounded for $\tau_+$, then for any $(x,v)\in SM$
	$$\varphi(T,x,v)\leq -\beta T< -D \leq \varphi(0,x,v).$$
	Since $\varphi$ is continuous, there exist $\delta>0$ and $-\beta T<\alpha_1<\alpha_2<-D$ such that
	$$\varphi(t,x,v)>\alpha_2,\quad \hbox{for } 0\leq t\leq \delta, \quad (x,v)\in SM;$$
	$$\varphi(t,x,v)<\alpha_1,\quad \hbox{for }T-2\delta\leq t\leq T, \quad (x,v)\in SM.$$ 
	
	Let $\chi\in C^\infty_0(\mathbb R)$ be a cut-off function, such that $0\leq \chi\leq 1$ and 
	\begin{align*}
		\chi(t)=\left\{\begin{array}{rcll}
			1,  &0\leq t\leq T-2\delta,\\
			0,  &T-\delta\leq t\leq T.
		\end{array}\right.
	\end{align*}
	Let $u(t,x,v)=\chi(t) \p_t f(t,x,v)$, then
	$$Pu=\chi \widetilde S \p_t S_0+\p_t \chi \p_t f.$$
	Moreover,  
	$u(T,x,v)=0$ and, from the transport equation,
	$$u(0,x,v)=\p_t f(0,x,v)=-Xf_0-\sigma f_0+\widetilde S(x,v) S_0(0,x,v).$$
	We apply Theorem \ref{Carleman estimate Riemannian} to $u$ and use $\varphi(t,x,v)\leq \varphi(0,x,v)$ for $t\geq 0$ in $SM$, 
	\begin{equation*}
		\begin{split}
			& s \int_{SM} e^{2s\varphi(0,x,v)}|-Xf_0-\sigma f_0+\widetilde S S_0(0,x,v)|^2\, d\Sigma\\
			\leq \,& C\int_0^T\int_{SM} e^{2s\varphi}|\chi\widetilde S\p_tS_0|^2\, d\Sigma dt+C\int_0^T\int_{SM} e^{2s\varphi}|\p_t\chi \p_t f|^2\, d\Sigma dt+s\int_0^T\int_{\p_+ SM} e^{2s\varphi}|\chi \p_t f|^2\, d\xi dt\\
			\leq \, & C\int_0^T\int_{SM} e^{2s\varphi(0,x,v)}|\widetilde S|^2\, d\Sigma dt+C\int_{T-2\delta}^{T-\delta}\int_{SM} e^{2s\varphi}|\p_t\chi \p_t f|^2\, d\Sigma dt+Ce^{Cs}\int_0^T\int_{\p_+ SM} |\p_t f|^2\, d\xi dt\\
			\leq \, & CT\int_{SM} e^{2s\varphi(0,x,v)}|\widetilde S|^2\, d\Sigma +Ce^{2s\alpha_1}\int_{T-2\delta}^{T-\delta}\int_{SM} |\p_t f|^2\, d\Sigma dt+Ce^{Cs}\int_0^T\int_{\p_+ SM} |\p_t f|^2\, d\xi dt.
		\end{split}
	\end{equation*}
	By Lemma \ref{energy estimate Riemannian} and $0<c_1\leq S_0(0,x,v)$, it follows that
	\begin{equation*}
		\begin{split}
			&  (s-CT) \int_{SM} e^{2s\varphi(0,x,v)}|\widetilde S|^2\, d\Sigma\\
			\leq \,& Cs\int_{SM} e^{2s\varphi(0,x,v)}|Xf_0+\sigma f_0|^2\,d\Sigma+Ce^{2s\alpha_1}\int_{T-2\delta}^{T-\delta}\int_{SM} |\p_t f|^2\, d\Sigma dt+Ce^{Cs}\int_0^T\int_{\p_+ SM} |\p_t f|^2\, d\xi dt\\
			\leq \,& Ce^{Cs}(\|Xf_0\|^2_{L^2(SM)}+\| f_0\|^2_{L^2(SM)}+\|\p_t f_-\|^2_{L^2(\p_-SM_T)})+Ce^{2s\alpha_1}\|\widetilde S\|^2_{L^2(SM)}+Ce^{Cs}\|\p_t f\|^2_{L^2(\p_+SM_T)}.
		\end{split}
	\end{equation*}
	Since $\varphi(0,x,v)>\alpha_2$, for $s$ large enough so that ${s\over 2}>CT$, we can derive that
	\begin{equation*}
		\begin{split}
			& {1\over 2} se^{2s\alpha_2}\|\widetilde S\|^2_{L^2(SM)}\leq  (s-CT) \int_{SM} e^{2s\varphi(0,x,v)}|\widetilde S|^2\, d\Sigma\\
			\leq \,& Ce^{Cs}(\|Xf_0\|^2_{L^2(SM)}+\| f_0\|^2_{L^2(SM)}+ \|\p_t f_-\|^2_{L^2(\p_-SM_T)})+Ce^{2s\alpha_1}\|\widetilde S\|^2_{L^2(SM)}+Ce^{Cs}\|\p_t f\|^2_{L^2(\p_+SM_T)}.
		\end{split}
	\end{equation*}
	Since $\alpha_2>\alpha_1$, we choose $s$ large enough such that ${s\over 2} se^{2s\alpha_2}-Ce^{2s\alpha_1}>0$   we have
	$$ \LC {s\over 2}e^{2s\alpha_2}-Ce^{2s\alpha_1}\RC\|\widetilde S\|^2_{L^2(SM)}\leq Ce^{Cs}(\|Xf_0\|^2_{L^2(SM)}+\| f_0\|^2_{L^2(SM)}+\|\p_t f_-\|^2_{L^2(\p_-SM_T)}+\|\p_t f\|^2_{L^2(\p_+SM_T)}).
	$$
	This completes the proof.
\end{proof}

\subsection{Reconstruction of the nonlinear term on a Riemannian Manifold}\label{sec:reconstruction Riemannian}
Let $f \equiv f(t,x,v;\varepsilon)$ be the solution to the problem 
\begin{align}\label{EQN: NRTE equ IP Riemannian}
	\left\{\begin{array}{rcll}
		\p_tf + X f + \sigma f + N(x,v,f) &=& 0& \hbox{in } SM_T,  \\
		f &=&\varepsilon h & \hbox{on }\{0\}\times SM,\\
		f &=& 0  & \hbox{on } \p_-SM_T.
	\end{array}\right.
\end{align}

With the help of the Carleman estimate and linearization techinique, we are ready to show the following two cases of $N$. 
\subsubsection{The case $N=\sum q^{(k)}{f^k\over k!}$}
We show the first main result in the Riemannian case. 
\begin{proof}[Proof of Theorem~\ref{thm:into Riemannian 1}]
    Consider $N(x,v,f)=\sum_{k=2}^\infty q^{(k)}(x,v){f^k\over k!}$, we can follow similar arguments as in Section~\ref{sec:sigma mu} and Section~\ref{sec:taylor} in the absence of the scattering term for the problem \eqref{EQN: NRTE equ IP Riemannian} to recover the unknown terms. In particular, we can deduce that $\sigma_1=\sigma_2$ and also $q_1^{(k)}=q_2^{(k)}$ in $SM$ so that $N_1(x,v,f)=N_2(x,v,f)$ by utilizing Theorem~\ref{Thm:estimate tilde S Riemann}. 
\end{proof}

\subsubsection{The case $N=qN_0(f)$}
Suppose that the nonlinear term has the form
\[
N(x,v,f)= q(x,v)N_0(f),
\] 
where $N_0$ satisfies
\begin{align}\label{EST:N term 4}
    \|N_0(f)\|_{L^\infty(SM_T)}\leq C_1 \|f\|_{L^\infty(SM_T)}^\ell,
\end{align} 
and
\begin{align}\label{EST:N 4}
	\|\p_z N_0(f)\|_{L^\infty(SM_T)}\leq C_2 \|f\|_{L^\infty(SM_T)}^{\ell-1}
\end{align}
for a positive integer $\ell\geq 2$ and constants $C_1,C_2>0$, independent of $f$. We can show as in the proof of Theorem~\ref{THM:WELL} that the well-posedness of \eqref{EQN: NRTE equ IP Riemannian} holds under the assumptions \eqref{EST:N term 4}-\eqref{EST:N 4}.

We will apply Theorem~\ref{Thm:estimate tilde S Riemann} to recover $\sigma$ and $q$. 
The strategy is to recover $\sigma$ by applying the first linearization and Theorem~\ref{Thm:estimate tilde S Riemann}. After that, we will employ the second linearization to recover $q$.
 
The first linearization of the problem \eqref{EQN: NRTE equ IP Riemannian} with respect to $\varepsilon$ at $\varepsilon=0$ is
\begin{align}\label{EQN: linear RTE R}
	\left\{\begin{array}{rcll}
		\p_tF^{(1)} + X F^{(1)} + \sigma F^{(1)} &=& 0& \hbox{in } SM_T,  \\
		F^{(1)} &=& h & \hbox{on }\{0\}\times SM,\\
		F^{(1)} &=& 0 & \hbox{on }\p_-SM_T.
	\end{array}\right.
\end{align}
where recall that $F^{(1)}(t,x,v) : =\partial_\varepsilon|_{\varepsilon=0} f(t,x,v;\varepsilon)$.
The problem now is reduced to studying the inverse coefficient problem for linear transport equations. 
\begin{proposition}\label{thm:recover sigma Riemannian} 
	Suppose that $\sigma_j\in L^\infty(SM)$ satisfies \eqref{EST:sigma} and $q_j\in L^\infty(SM)$ for $j=1,\,2$. Let $h\in L^\infty(SM)$ satisfy $0<c_1\leq h\leq c_2$ for some positive constants $c_1,\,c_2$ and $X^\beta h\in L^\infty(S\Omega)$, $\beta=1,\,2$.
    Let $f_j$ be the solution to the problem \eqref{EQN: NRTE equ IP Riemannian} with $\sigma$ replaced by $\sigma_j$ and $q$ replaced by $q_j$, and $F^{(1)}_j=\p_\varepsilon|_{\varepsilon=0} f_j$, $j=1,\,2$.
	Then
	\begin{align}\label{EST:q}
		\|\sigma_1-\sigma_2\|_{L^2(SM)}\leq C \|\p_t F_1^{(1)}-\p_tF_2^{(1)}\|_{L^2(\p_+SM_T)}
	\end{align} 
	for some constant $C>0$.
\end{proposition}

\begin{proof}
	To obtain \eqref{EST:q}, let $w^{(1)}:=F_1^{(1)}- F_2^{(1)}$ and then $w$ is the solution to   
	\begin{align*}
		\left\{\begin{array}{rcll}
			\p_t w^{(1)} + X w^{(1)} + \sigma_1 w^{(1)}  &=& -(\sigma_1-\sigma_2)F_2^{(1)}& \hbox{in } SM_T,  \\
			w^{(1)} &=&0 & \hbox{on }\{0\}\times SM,\\
			w^{(1)} &=& 0 & \hbox{on } \p_-SM_T.
		\end{array}\right.
	\end{align*}
	Note that $F^{(1)}_2|_{t=0}=h$ is strictly positive in $SM$ and also satisfies $\| F^{(1)}_2\|_{L^\infty(SM_T)}, \|\p_t F^{(1)}_2\|_{L^\infty(SM_T)}\leq C$ for some constant $C>0$.
	By Theorem~\ref{Thm:estimate tilde S Riemann}, it immediately implies \eqref{EST:q}. The proof is complete.
\end{proof}

With the establishment of Proposition~\ref{thm:recover sigma Riemannian}, if $\mathcal{A}_{\sigma_1,N_1}=\mathcal{A}_{\sigma_2,N_2}$, then $\sigma_1=\sigma_2$. Hence we let $\sigma:=\sigma_1=\sigma_2$. Now we will recover $q$ in the nonlinear term. 
To do so, we apply the linearization again at $\varepsilon=0$ to obtain 
\begin{align}\label{EQN: linear RTE R 1}
	\left\{\begin{array}{rcll}
		\p_tF^{(2)}_j + X F^{(2)}_j + \sigma F^{(2)}_j &=& -q_j(x,v)  \p_\varepsilon^2|_{\varepsilon=0}N_0(f_j) & \hbox{in } SM_T,  \\
		F^{(2)}_j &=& 0 & \hbox{on }\{0\}\times SM,\\
		F^{(2)}_j &=&0 & \hbox{on }\p_-SM_T.
	\end{array}\right.
\end{align}
Since $f_j|_{\varepsilon=0}=0$, by assumption \eqref{EST:N 4}, we get $\p_z N_0(0)=0$ and
$$\p^2_\varepsilon|_{\varepsilon=0} N_0(f_j)=\p_z^2 N_0(0)\, (F_j^{(1)})^2.$$
Notice that since $F^{(1)}_j$, $j=1,\,2$ satisfy the same transport equation with the same initial data and boundary data on $\p_-SM_T$. The well-posedness theorem yields that
$$
F^{(1)} :=F^{(1)}_1=F^{(1)}_2.
$$
Therefore we denote $M_0(F^{(1)}):=\p_\varepsilon^2|_{\varepsilon=0}N_0(f_j)=\p_z^2 N_0(0)\, (F^{(1)})^2$ for $j=1,\,2$.
The function $w^{(2)}:= F^{(2)}_1-F^{(2)}_2$ satisfies
\begin{align}\label{EQN: linear RTE R 2}
	\left\{\begin{array}{rcll}
		\p_tw^{(2)}  + X w^{(2)}  + \sigma w^{(2)}  &=& -(q_1-q_2)(x,v) M_0(F^{(1)}) & \hbox{in } SM_T,  \\
		w^{(2)}  &=& 0 & \hbox{on }\{0\}\times SM,\\
		w^{(2)} &=&0 & \hbox{on }\p_-SM_T.
	\end{array}\right.
\end{align}
According to Theorem~\ref{Thm:estimate tilde S Riemann}, we then have the second main result in the Riemannian case.

\begin{proposition}\label{thm:recover q Riemannian} 
	Suppose that $\sigma\in L^\infty(SM)$ satisfies \eqref{EST:sigma} and $q_j\in L^\infty(SM)$ for $j=1,\,2$. Let $h\in L^\infty(SM)$ satisfy $0<c_1\leq h\leq c_2$ for some positive constants $c_1,\,c_2$ and $X^\beta h\in L^\infty(SM)$, $\beta=1,\,2$.
	Suppose that the nonlinear term satisfies $\p_z^2 N_0(0)>0$.
	Let $f_j$ be the solution to the problem \eqref{EQN: NRTE equ IP Riemannian} with $q$ replaced by $q_j$, and $F^{(2)}_j=\p^2_\varepsilon|_{\varepsilon=0} f_j$, $j=1,\,2$.
	Then
	\begin{align}\label{EST:q M}
		\|q_1-q_2\|_{L^2(SM)}\leq C \|\p_t F^{(2)}_1-\p_t F^{(2)}_2\|_{L^2(\p_+SM_T)}
	\end{align} 
	for some constant $C>0$ depending on $\sigma,\,h$ and $N_0$.
	
	Moreover, if $\mathcal{A}_{\sigma_1,N_1}(f_0,0)=\mathcal{A}_{\sigma_2,N_2}(f_0,0)$ for any $(f_0,0)\in \mathcal{X}^M_\delta$, then
$$
q_1=q_2\quad\hbox{in }SM.
$$
\end{proposition}
\begin{proof}
Since $0<c_1 \leq F^{(1)}(0,x,v)=h\leq c_2$ and $\p_z^2 N_0(0)>0$, we have that $0\leq \tilde c_1\leq M_0(F^{(1)})\leq \tilde c_2$ for some positive constants $\tilde c_1, \tilde c_2$. One the other hand, by applying Proposition \ref{prop:forward problem linear} to the problem \eqref{EQN: linear RTE R}, we obtain that $\|F^{(1)}\|_{L^\infty(SM_T)}\leq C\|h\|_{L^\infty(SM)}$, and $\|\p_t F^{(1)}\|_{L^\infty(SM_T)}\leq C\|Xh+\sigma h\|_{L^\infty(SM)}\leq C(\|Xh\|_{L^\infty(SM)}+\|h\|_{L^\infty(SM)})$. Now we can directly apply Theorem \ref{Thm:estimate tilde S Riemann} to finish the proof.
\end{proof}

\begin{proof}[Proof of Theorem~\ref{thm:into Riemannian}]
Combining Proposition~\ref{thm:recover sigma Riemannian} and Proposition~\ref{thm:recover q Riemannian}, we have the unique determination of $\sigma$ and $N$ from the boundary data. 
\end{proof}

Finally we end this section by proving Theorem~\ref{thm:into Euclidean N2}, where the transport equation $\p_t f+v\cdot \nabla_x f+\sigma f+N(x,v,f)=K(f)$ in $\Omega$ with $N$ defined as $N(x,v,f)=q(x,v)N_0(f)$, which is different from the setting in Section~\ref{sec:Euclidean}.

\begin{proof}[Proof of Theorem~\ref{thm:into Euclidean N2}]
By utilizing the techniques here and also in Section~\ref{sec:Euclidean} (in particular, Theorem~\ref{thm:recover source Carleman}), we can conclude the following two results: 
(1) If $\mu$ is given, then $\sigma$ and $N$ are uniquely determined by the boundary data $\mathcal{A}_{\sigma,\mu,N}$. 
    
(2) On the other hand, if $\sigma$ is given, then $\mu$ and $N$ are uniquely determined by the boundary data $\mathcal{A}_{\sigma,\mu,N}$. 
\end{proof}

\vskip1cm
\textbf{Conflict of Interest.}
Ru-Yu Lai is partially supported by the National Science Foundation through grant DMS-2006731. Gunther Uhlmann is partly supported by NSF, a Simons Fellowship, a Walker Family Professorship at UW, and a Si Yuan Professorship at IAS, HKUST. Hanming Zhou is partly supported by the NSF grant DMS-2109116.

\vskip.5cm
\textbf{Data Availability.}
Data sharing not applicable to this article as no datasets were generated or analysed during the current study.

\bibliographystyle{abbrv}
\bibliography{reference}

\end{document}